\documentclass[reqno,12pt,letterpaper]{amsart}
\usepackage{amsmath,amssymb,amsthm,graphicx,mathrsfs,url}
\usepackage[usenames,dvipsnames]{color}
\usepackage[colorlinks=true,linkcolor=Red,citecolor=Green]{hyperref}

\usepackage{graphicx}
\usepackage[colorlinks=true]{hyperref}
\usepackage{color}
\usepackage{amsmath,amsfonts}
\numberwithin{equation}{section}
\usepackage{amsmath, amsthm}
    \usepackage{dsfont}
    \usepackage{fancybox}
    \usepackage{amssymb}

\newcommand{\Ss}{\mathbb{S}}

\newcommand{\R}{\mathbb{R}}

\newcommand{\Sym}{\text{Sym}}

\newcommand{\disp}[1]{\displaystyle{#1}}

\newcommand{\epsi}{\varepsilon}
\newcommand{\LL}{{\mathcal{L}}}
\newcommand{\1}{\mathds{1}}
\newcommand{\al}{\al}

 \newcommand{\Id}{\mathrm{Id}}
\newcommand{\te}{\theta}
\newcommand{\lr}[1]{\langle #1 \rangle}
\newcommand{\FF}{\mathcal{F}}
\newcommand{\GG}{\mathcal{G}}

\newcommand{\PP}{\mathcal{P}}
\newcommand{\supp}{\mathrm{supp}}

\newcommand{\Z}{\mathbb{Z}}

\newcommand{\Aff}{\text{Aff}}
\newcommand{\matrice}[1]{\left(\begin{matrix}
#1
\end{matrix} \right)}

\newcommand{\spann}{\text{span}}
\newcommand{\TT}{\mathcal{T}}

\newcommand{\EE}{\mathcal{E}}
\newcommand{\Ii}{\mathcal{I}}

\newcommand{\MM}{\mathcal{M}}
\newcommand{\SSS}{\mathcal{S}}

\newcommand{\tT}{\widetilde{T}}
\newcommand{\HH}{\mathcal{H}}

\newcommand{\az}{\alpha}

\newcommand{\lv}{\left\vert}
\newcommand{\rv}{\right\vert}

\newcommand{\Hh}{\mathbb{H}}

\newcommand{\DD}{\mathcal{D}}
\newcommand{\PPP}{\mathcal{P}}

\newcommand{\system}[1]{\left\{\begin{matrix}
#1
\end{matrix} \right.}
\newcommand{\RR}{\mathcal{R}}

\newcommand{\Gg}{\mathbb{G}}

\newcommand{\gtu}{g_n^1}
\newcommand{\gtd}{g_n^2}
\newcommand{\RRR}{\mathcal{R}}

\title[Quantitative form of certain $k$-plane transform inequalities.]{Quantitative form of certain $k$-plane transform inequalities.}

\author{Alexis Drouot}
\email{alexis.drouot@gmail.com}
\address{Department of mathematics, University of California, Berkeley, CA.}

\newtheorem{thm}{Theorem}
\newtheorem{cor}[thm]{Corollary}
\newtheorem{lem}{Lemma}[section]
\newtheorem{proposition}[lem]{Proposition}

\newtheorem{theorem}[thm]{Theorem}

\begin{document}

\begin{abstract} Let $d \geq 2$ and $1 \leq k \leq d-1$. The $k$-plane transform satisfies some $L^p \rightarrow L^q$ dilation-invariant inequality. In this case the best constant and the extremizers are explicitly known. We give a quantitative form of the inequality with respect to these extremizers, that works for $k=d-1$ and for $k \neq d-1$ while restricted to radial functions.
\end{abstract}

\maketitle

\section{Introduction and statement of results}
Consider $d \geq 2$, $1 \leq k \leq d-1$ and define $\GG_{k,d}$ the set of $k$-planes in $\R^d$ passing through the origin and $\MM_{k,d}$ the set of affine $k$-planes in $\R^d$. If $f$ is a smooth compactly supported function on $\R^d$, we define the $k$-plane transform of $f$ as the map
\begin{equation*}
\RR f : \MM_{k,d} \rightarrow \R
\end{equation*}
defined by
\begin{equation*}
\RR f(\Pi) = \int_\PP f(x) d\lambda_\Pi(x)
\end{equation*}
where $d\lambda_\Pi$ is the measure on $\Pi$ induced by the Lebesgue measure on $\R^d$.\\ 

Provide the space $\GG_{k,d}$ with its unique rotation-invariant probability measure. This measure induces a unique translation-invariant measure on $\MM_{k,d}$ -- for details see \cite{Mattila}. The question of the continuity of $\RR$ on Lebesgue spaces has been discussed by many authors, see \cite{ObeSte}, \cite{Drury1}, \cite{Christ1}, \cite{Drury}. The optimal result, obtained in \cite{Christ1} and whose proof was simplified in \cite{Drury}, is the following. Define
\begin{equation}\label{eq:pq}
p = \dfrac{d+1}{k+1}, \ \ q = d+1.
\end{equation} 
The $k$-plane transform extends to a continuous operator from $L^1(\R^d)$ to $L^1(\MM_{k,d})$ and from $L^p(\R^d)$ to $L^q(\MM_{k,d})$. Every other extension on Lebesgue spaces is obtained from interpolation between these two.\\

Baernstein and Loss formulated in \cite{BaeLos} a conjecture concerning the norm of $\RR$ on these spaces. The conjecture has been answered positively in the endpoint case, see \cite{Christ4} for $k=d-1$ and \cite{Drouot} for any value of $k$. The result in \cite{Christ4} result gives the complete characterization of extremizers while the method employed in \cite{Drouot} does not. However, Drouot's result was nicely completed by \cite{Flock} in the following theorem:

\begin{theorem}\label{thm:mapR} (Christ, Drouot, Flock) For $p,q$ given by \eqref{eq:pq} the following inequality holds:
\begin{equation}\label{eq:mapR}
\dfrac{|\RR f|_q}{|f|_p} \leq A_\RR, \ \ A_\RR =  \left(2^{k-d} \dfrac{|\Ss^k|^d}{|\Ss^d|^k} \right)^\frac{1}{d+1}.
\end{equation}
Moreover, equality is realized if and only if $f$ has the form
\begin{equation}\label{eq:hg}
f(x) = C\left(\dfrac{1}{1+|Lx|^2} \right)^\frac{k+1}{2}
\end{equation}
where $L$ is any invertible affine map on $\R^d$, and $C$ is a nonzero constant. 
\end{theorem}

Equation \eqref{eq:hg} proves that the group 
\begin{equation*}
\Gg = \R-\{ 0 \} \times \left( \text{Gl}(\R^d)/O(d)\right) \ltimes \R^d
\end{equation*}
acts freely and transitively on extremizers. Here $\text{Gl}(\R^d)$ is the group of all linear invertible maps on $\R^d$, $O(d)$ is the subgroup of all orthogonal map, and the action $\ltimes$ is induced while seeing the group $\Aff(\R^d)$ of invertible affine maps on $\R^d$ as $\text{Gl}(\R^d) \ltimes \R^d$. If $\psi = (C,L) \in \Gg$ and $f \in L^p(\R^d)$ we define
\begin{equation*}
\psi \star f = C|\det(L)|^{1/p} f \circ L
\end{equation*}
that is, the range of an element of $L^p$ under the natural action of an element of $\Gg$. We also define $h_0$ the normalized extremizer given by
\begin{equation}\label{eq:extrpa}
h_0(x) = c_d\left(\dfrac{1}{1+|x|^2} \right)^{\frac{k+1}{2}},  \ \ \ c_d = \left( |S^d| \right)^{-1/p}.
\end{equation}
This extremizer has $L^p(\R^d)$-norm $1$.

The fact that it has been possible to identify all extremizers relies heavily on the presence of a high-dimension group of symmetries. The approach developed in \cite{Christ4} and \cite{Drouot}, although quite different, cannot be adapted to the interpolated inequalities: they are not conformally invariant. The conformal invariance of the inequality \ref{eq:mapR} was expressed in \cite{Drury}, and rediscovered in \cite{Christ4} -- we will give more details later. We happen to be in the same mysterious situation of the Hardy-Littlewood-Sobolev inequality: while a lot of authors found by different methods the sharp inequality in the conformally invariant case, none was able to find the best constant for the non-conformally invariant case (i.e, the interpolated inequalities). Even in restricted particular cases it would be of outstanding mathematical interest if someone could solve this long-standing problem!\\

In this paper, we are interested in the stability of Theorem \ref{thm:mapR}. More precisely, if $f \in L^p(\R^d)$ has norm $1$, and if $|\RR f|_q$ has large norm, how close is $f$ from extremizers? This question already has been answered in a non-quantitative way in \cite{Christ3} for $k=d-1$. In this paper we aim to respond to this question in a quantitative way. It should be noted that even the non-quantitative result is far from being a trivial problem: if $h_0$ is an extremizer and $L_n \in \Aff(\R^d)$ form an unbounded sequence, then $L_n \star h_0$ has no chance to converge in $L^p$. 

\subsection{Main results} Our results are the following:

\begin{theorem}\label{thm:global} There exists a constant $c_\RRR$ such that the following is valid. Assume that $k=d-1$, or that $f$ is a radial function. Define
\begin{equation*}
d(f,\EE) = \inf_{h \in \EE} |f - h|_p.
\end{equation*}
Then
\begin{equation*}
\dfrac{|\RRR f|_q}{|f|_p}\leq A_\RRR \left( 1 - c_\RRR \left( \dfrac{d(f,\EE)}{|f|_p} \right)^2 \right) .
\end{equation*}
\end{theorem}

Theorem \ref{thm:global} will follow from Theorem \ref{thm:precompact} and Theorem \ref{thm:local} below and from the non-quantitative analysis developed in \cite{Christ3}. We say that a sequence of functions $f_n \in L^p(\R^d)$ is an extremizing sequence for \eqref{eq:mapR} if it satisfies 
\begin{equation*}
|f_n|_p = 1, \ \ \lim_{n \rightarrow \infty} |\RR f_n|_q = A_\RR.
\end{equation*}

\begin{theorem}\label{thm:precompact} Let $f_n$ be an extremizing sequence of radial functions. Then there exists a subsequence $f_{\varphi(n)}$ and a sequence of numbers $\lambda_{\varphi(n)}$ such that the sequence
\begin{equation*}
\lambda_{\varphi(n)}^{d/p} f\left(\lambda_{\varphi(n)} \cdot \right)
\end{equation*}
converges in $L^p(\R^d)$ to $h_0$.
\end{theorem}

As $f_n$ is assumed to be radial, and as all radial extremizers are obtained from $h_0$ under dilations, it suffices to prove that that $f_n$ converges in $L^p$ after rescaling and extraction of a subsequence. This theorem is non quantitative in the following way: it says that there exists a function $\epsi \mapsto \delta(\epsi)$ defined near $0$, with 
\begin{equation*}
\lim_{\epsi \rightarrow 0} \delta(\epsi) = 0
\end{equation*}
such that for $\epsi$ small enough
\begin{equation*}
\dfrac{|\RR f|_q}{|f|_p} \geq A_\RR(1-\epsi)
\end{equation*}
implies 
\begin{equation*}
\dfrac{d(f,\EE)}{|f|_p} \leq \delta(\epsi). 
\end{equation*}
In the Theorem below we prove that we can actually take $\delta(\epsi) \sim \epsi^{1/2}$.

\begin{theorem}\label{thm:local} There exist some positive constants $\lambda_\RRR, \epsi_0$ such that the following is valid. For all $1 \leq k \leq d-1$,
\begin{equation*}
\dfrac{|\RRR f|_q}{|f|_p}\leq A_\RRR \left( 1 - \lambda_\RRR \left( \dfrac{d(f,\EE)}{|f|_p} \right)^2 + o\left(\left( \dfrac{d(f,\EE)}{|f|_p} \right)^{2+\epsi_0}\right)\right).
\end{equation*}
\end{theorem}

Note that this theorem is valid without any radial assumption. The missing ingredient to remove the radial restriction for $k \neq d-1$ in Theorem \ref{thm:global} is a precompactness theorem for general extremizing sequences, as proved in \cite{Christ3}. However, Theorem \ref{thm:global} might be useful as a tool in itself to prove a more general statement -- see the approach developed in \cite{CiaFusMagPra} for the Sobolev inequality.\\

We end this introduction by some historical references. Since the pioneering work of Bianchi and Egnell \cite{BiaEgn} for the Sobolev inequality, statements like Theorem \ref{thm:global} have received a lot of interest. Among other, see \cite{CheFraWet} for a very clear and enlightening treatment of the fractional Sobolev inequality, \cite{CiaFusMagPra} for an extension of Bianchi and Egnell's result that does not require a precompactness result, and \cite{Christ5} for an outstanding treatment of the Hausdorff-Young inequality elegantly using deep combinatorics theorems.

The paper is divided into two fairly independent parts. The first one, from Section \ref{sec2} to \ref{sec:4}, deals with the proof of Theorem \ref{thm:precompact}. The second one, from Section \ref{sec:12} to \ref{sec:15}, deals with the proof of Theorems \ref{thm:global} and \ref{thm:local}.

\subsection{Outline of the proof of Theorem \ref{thm:precompact}.}
As carefully explained in \cite{Drouot}, restraining the inequality \eqref{eq:mapR} to radial functions is equivalent to consider the simpler one-dimensional operator $\TT$, formally defined as
\begin{equation*}
\TT f(r) = \disp{ \int_0^\infty f(\sqrt{r^2+s^2}) s^{k-1}ds = \int_r^\infty f(u) (u^2-r^2)^{k/2-1} u du}.
\end{equation*}
This operator maps $L^p((0,\infty),r^{d-1} dr)$ to $L^q((0,\infty),r^{d-k-1}dr)$. The corresponding inequality is
\begin{equation}\label{eq:map2}
| \TT f |_{L^q((0,\infty),r^{d-k-1}dr)} \leq A_\TT | f |_{L^p((0,\infty),r^{d-1 }dr)}
\end{equation}
for an optimal constant $A_\TT$ that has been computed in \cite{Drouot}. The group of dilation is a non-compact group preserving \eqref{eq:map2}. Although $\TT$ enjoys a much simpler behavior than the $k$-plane transform, the reader should keep in mind its connection to $\RRR_k$. The $k$-plane transform has a geometric origin that provides a nice and simple geometric interpretation of the action of $\TT$. This has guided our intuition all along the coming Lemma and proofs. Theorem \ref{thm:precompact} is rigorously equivalent to:

\begin{theorem}\label{thm:precompact2} 
Let $1 \leq k \leq d-1$ and $f_n$ be an extremizing sequence for \eqref{eq:map2}. Then $f_n$ is relatively compact in $L^p(r^{d-1}dr)$ modulo the group of dilations, that is, there exists a sequence $\lambda_{\varphi(n)} \in \R$ such that the rescaled sequence 
\begin{equation*}
\lambda_\varphi(n)^{d/p} f_{\varphi(n)}(\lambda_{\varphi(n)} \cdot )
\end{equation*} 
converges strongly in $L^p(\R^d)$.
\end{theorem}

Since Christ thoroughly studied the case $k=d-1$ in \cite{Christ3} and \cite{Christ4} in the proof we will restrict our attention to the case $d \geq 3$.\\

The strategy is outlined as follows. Everything must be understood modulo successive extractions of subsequences. At first we focus on proving that every nonincreasing extremizing sequence converges (after appropriate and necessary rescaling) to an extremizer. Such an analysis finds its origin in \cite{Lieb}.

Let $f_n$ be \textit{any} radial extremizing sequence. $f_n$ admits a weak limit, called $f$, and our goal is to prove that $f$ is actually a strong limit. We thus would like to use a compact operator to transform this weak limit into a strong limit. Of course, here the natural operator to use is $\TT$. We are missing an assumption: $\TT$ as defined earlier is not compact. Its kernel
\begin{equation*}
k(r,s) = \1_{s \geq r} (s^2-r^2)^{k/2-1} s
\end{equation*}
admits essentially two singularities, one for $r=s$ when $k=1$, and another one for $s=\infty$. As a consequence, we define a truncated operator,
\begin{center}
$\begin{array}{ccccc}
& \TT_R : & L^\infty([0,R]) & \longrightarrow & L^q([0,R])\\
& & f & \longmapsto & \TT \1_{[0,R]} f = \TT f. \end{array}$
\end{center}
In this setting, since we artificially erased the singularities, for any $0<R < \infty$, $\TT_R$ is compact -- see the appendix.\\

Write now $f_n = g_n^m + \epsi_n^m$ with $g_n^m := \1_{[0,m]} \1_{\{ f_n \leq m\}} f_n$. We want to take the limit of this equality as $n,m$ tends to infinity. Let $\TT$ acts on this equality:
\begin{equation*}
\TT f_n = \TT_m g_n^m + \TT \epsi_n^m.
\end{equation*}
It would be convenient to be able to take the limit $n \rightarrow \infty$ \textit{then} $m \rightarrow \infty$ for $\TT_m g_n^m$ and  $m \rightarrow \infty$ \textit{then} $n \rightarrow \infty$ for $\TT \epsi_n^m$. This is possible if we can prove the two following facts:
\begin{enumerate}
\item[$(i)$] $h^m$ converges in $L^q$ as $m \rightarrow \infty$, where $h^m$ is the strong limit of $\TT_m g_n^m$; 
\item[$(ii)$] $\epsi_n^m$ is uniformly small in $n$ as $m$ tends to infinity. 
\end{enumerate}
$(i)$ is the easiest, it follows from the structure of $\TT$. It will be proved in section \ref{sec:4}. $(ii)$ is actually much harder and will require two uniform bounds on $f_n$. These bounds can be obtained only after an appropriate rescaling of the sequence $(f_n)$. We split the proof of $(ii)$ into two parts:
\begin{enumerate}
\item In section \ref{sec2}, we rescale our extremizing sequence to make it converge weakly to a non-zero function. We will deduce the bound
\begin{equation*}
\disp{ \lim_{m \rightarrow \infty} \int_{|f_n| \geq m} |f_n|^p = 0 },
\end{equation*}
uniformly in $n$. This bound means that the rescaled sequence has no part that could converge to some kind of Dirac distribution.
\item In section \ref{sec:3} we use the concentration-compactness principle of Lions \cite{lions} to prove a strong, and uniform in $n$, localization result on the extremizing sequence, expressed as
\begin{equation*}
\disp{ \lim_{m \rightarrow \infty} \int_{r \geq m} |f_n|^p = 0 }.
\end{equation*}
\end{enumerate}

\subsection{Outline of the proofs of Theorem \ref{thm:local} and \ref{thm:global}} The natural approach to Theorem \ref{thm:local} is to write a second variation formula at an extremizer of the functional
\begin{equation*}
f \mapsto \dfrac{|\RRR f|_q}{|f|_p}
\end{equation*} and hope that the Hessian is definitive negative. As it stands this approach is a little too naive: the extremizers of \eqref{eq:mapR} are not isolated points in $L^p(\R^d)$ and thus the Hessian is only nonnegative. A less naive approach consists in saying that the Hessian is negative on the set orthogonal to extremizers -- whatever this means in $L^p$ when $p \neq 2$. This is, in a local sense precised in Section \ref{sec:15}, correct.\\

Our analysis consists in three steps. In Section \ref{sec:12}, we first recall a very interesting result from \cite{Drury} that expresses the conformal invariance of the inequality \eqref{eq:mapR}. Using a generalized Riemann projection we can see \eqref{eq:mapR} as a rotation-invariant inequality for functions on Grassmanian manifolds. Using Theorem \ref{thm:mapR} we give the sharp form of this inequality. We identify all extremizers, which form a smooth manifold of dimension
\begin{equation*}
\dfrac{(d+1)(d+2)}{2}.
\end{equation*}
This manifold contains the constant function $\1$, that is by its simplicity of special interest. We identify the tangent space $T_\1\EE$ at this function, which corresponds to first and second order harmonics on the Grassmannian $\GG_{1,d+1}$. The space $T_\1 \EE$ consists of functions lying in the Hilbert space $\HH = L^2(\GG_{1,d+1})$.\\

In Section \ref{sec:13}, we introduce the quadratic form $Q_\1$ we must study to prove \ref{thm:local}. This form -- defined on $\HH$ -- induces a bounded selfadjoint operator $\LL$ on $\HH$ that is rotation invariant. As $\GG_{1,d+1}$ is a homogeneous space it is a function of the Laplacian $\Delta_{\GG_{1,d+1}}$ and thus we can perform a complete spectral study of $\LL$. In particular using the minmax principle we will be able to prove that $Q_\1$ is negative on $T_\1 \EE^\perp$, the orthogonal of $T_\1 \EE$ in $\HH$.\\

In Section \ref{sec:14} we use results of \cite{Christ5} to rule out some technical difficulty: if $p < 2$, $L^p(\GG_{1,d+1})$ cannot be seen as a subspace of $\HH$. This induces a smoothness defect in the Taylor development of the function
\begin{equation*}
t \mapsto \dfrac{|\RR(h_0+tg)|_q}{|h_0+tg|_p}.
\end{equation*}
Because of Christ's result,this defect happens to be always of the right sign. It thus happens not to be an obstacle for our analysis.\\

In the last section we prove Theorems \ref{thm:global} and \ref{thm:local}.

\subsection{Notations} 
\begin{itemize}
\item  Let $A$ and $B$ be positive functions and $P$ be some statement. We will say that $P$ implies that $A \lesssim B$ when there exists a universal constant $C$, which depends only on $k$ and $d$, such that $P$ implies that $A \leq C B$. $A \gtrsim B$ will be the convert and $A \sim B$ will be used when $A \lesssim B$ and $B \lesssim A$.
\item $| f |_p$ denotes the $L^p$-norm of $f$, with respect to a contextual measure.
\item For $f$ measurable, we denote $f^\ast$ the radial, nonincreasing rearrangement of $f$, with respect to a contextual measure. It is the only radial nonincreasing function whose level sets have the same measure as the level sets of $f$, see \cite{Grafakos} for a more complete introduction.
\item We will denote $|E|$ the measure of a set $E$, with respect to a measure depending on the context.
\item For two sets $E, F \subset \R$, $d(E,F)$ denotes the standard distance between $E$ and $F$. If $R >0$, $d(R,E)$ denotes the distance between $[0,R]$ and $E$.
\item $\Aff(\R^d)$ is the group of all affine maps $\R^d \rightarrow \R^d$.
\item $\Sym(\R^d)$ is the set of linear symmetric maps of $\R^d$.
\item The Japanese bracket of $x \in \R^d$ is $\lr{x} = \sqrt{1+|x|^2}$.
\end{itemize}

Our purpose is to prove Theorem \ref{thm:precompact2} for \textit{a subsequence of an extremizing sequence $f_n$}. To simplify the notations, we will still call $f_n$ any subsequence that is extracted from $f_n$.

\subsection*{Acknowledgment} The author would like to thanks Professor Micha\"el Christ for stimulating discussions, especially for pointing out \cite{Drury1} and \cite{Christ5}. The author would also like to thanks the anonymous referee for suggesting that Theorem \ref{thm:precompact} should imply a quantitative version of Theorem \ref{thm:mapR}.

\section{Precompactness of radial, extremizing sequences.}\label{sec:2'} 
Let us recall a few things. The distribution function of $f \in L^p$ is defined by
\begin{equation*}
df(t) = \left| \left\{ |f| \geq t \right\} \right|.
\end{equation*}
If $1 \leq a,b \leq \infty$, we define $|\cdot|_{a,b}$ the Lorentz norm of order $(a,b)$ as 
\begin{equation*}
|f|_{a,b} =  \left(  b \int_0^\infty  \left(df(t)^{1/a} t \right)^b \dfrac{dt}{t} \right)^{1/b}, 
\end{equation*}
with obvious correction if $a$ or $b$ is infinity. The Lorentz space $L^{a,b}$ is simply the space of functions having finite Lorentz norm of order $(a,b)$. See \cite{Grafakos} for a more complete introduction.

Because of \cite{Christ1}, Theorems A and B, the operator $\TT$ is continuous from $L^{p,q}$ to $L^q$. Consider $f$ a $L^p$-function satisfying 
\begin{equation*}
|f|_p = 1, \ \ |\TT f|_q \gtrsim 1.
\end{equation*}
Then
\begin{equation*}
1 \lesssim |\TT f|_q \lesssim |f|_{p,q} \lesssim |f|_p^{p/q} |f|_{p,\infty}^{1-p/q} = |f|_{p,\infty}^{1-p/q}.
\end{equation*}
In the last inequality we used interpolation between Lorentz spaces. It follows that
\begin{equation}\label{eq:kvc}
|f|_{p,\infty} = \sup_{t > 0} t df(t)^{1/p} \gtrsim 1.
\end{equation}
We are now ready to prove the following:

\begin{proposition}\label{prop:pni} Assume that $f_n$ is an extremizing sequence of nonincreasing functions for \eqref{eq:map2}. Then modulo extraction of a subsequence there exists $\lambda_n$ such that
\begin{equation*}
\lambda_n^{d/p} f_n(\lambda_n \cdot)
\end{equation*}
converges in $L^p$.
\end{proposition}

\begin{proof} As $f_n$ is an extremizing sequence by \eqref{eq:kvc} there exists $t_n$ such that
\begin{equation*}
t_n df_n(t_n)^{1/p} \gtrsim 1.
\end{equation*}
We recall that $df(t)$ is computed using the measure $r^{d-1}dr$. Using that $f_n$ is nonincreasing it follows that there exists $\lambda_n$ such that
\begin{equation*}
f_n \gtrsim \lambda_n^{-d/p} \1_{[0,\lambda_n]}.
\end{equation*}
Define $g_n = \lambda_n^{d/p} f_n(\lambda_n\cdot)$. Then $g_n \gtrsim \1_{[0,1]}$ and therefore it has admits a subsequence that converges to a non-zero function $g \in L^p$. $g_n$ is an extremizing sequence of nonincreasing functions for \eqref{eq:map2}. By Helly's principle we can assume that both $g_n$ and $\TT g_n$ (which is nonincreasing because of the structure of $\TT$) converge almost everywhere respectively to $g$ and $\TT g$. Apply \cite{Lieb}, Lemma $2.7$ to conclude: $g_n$ converges in $L^p$ to an extremizer.
\end{proof}

The next sessions are devoted to the much harder case where $f_n$ is not assumed to be nonincreasing.

\section{A concentration-compactness result}\label{sec2}
We first want to prove the following: 

\begin{lem}\label{lem:fund} Let $f_n$ be an extremizing sequence for \eqref{eq:map2}. There exists $R_0$ such that the following is satisfied.  For all $n$, there exist a set $E_n \subset \R$, and $\lambda_n \in \R$, such that if
\begin{equation*}
g_n = \lambda_n^{d/p} \star f_n(\lambda_n\cdot),
\end{equation*}
then:
\begin{enumerate}
\item $|E_n| \sim 1$; 
\item $g_n \geq \1_{E_n}$;
\item $E_n \subset [0,R_0]$.
\end{enumerate}
\end{lem}

This lemma is an improvement of \eqref{eq:kvc}. There are essentially three things to identify: $\lambda_n$, $E_n$, and $R_0$. 

\begin{proof} Consider $f_n^\ast$ the radial nonincreasing rearrangement of $f_n$. Using Theorem D in \cite{Christ1}, $f_n^\ast$ is still an extremizing sequence. By Proposition \ref{prop:pni} there exists $\lambda_n$ such that
\begin{equation*}
\lambda_n^{d/p} f_n^\ast(\lambda_n \cdot) 
\end{equation*}
converges to an extremizer. Define now $g_n = \lambda_n^{d/p} f_n(\lambda_n \cdot)$. As $g_n$ and $g_n^\ast$ have the same $L^{p,\infty}$-norm and $g_n^\ast$ converges to an extremizer there exists a set $\EE_n$ with measure of order $1$ such that $g_n \geq \1_{\EE_n}$. Denote $e=|\EE_n|>0$. This number $e$ can be chosen independently of $n$ since a lower bound on $|\EE_n|$ exists.\\

All that remain to be shown is that the sets $\EE_n$ have most of their weight uniformly close from 0. Let us chose $R >0$ and call $F_n = \EE_n \setminus [0,R]$, $\delta_n(R) = |F_n|$. Assume $\liminf_n \delta_n(R)$ is bounded from below when $R \rightarrow \infty$. Then $g_n$ cannot be an extremizing sequence. Indeed consider $ h_n = g_n - \1_{F_n} $. We have:
\begin{equation*}
| h_n |_p^p = \disp{ \int_0^\infty [g_n-\1_{F_n}]^p \leq \int_0^\infty g_n^p - \1_{F_n} \leq 1-\delta_n(R) }.
\end{equation*}
On the other hand,
\begin{equation*}
| \TT h_n |_q \geq | \TT g_n |_q - | \TT \1_{F_n} |_q.
\end{equation*}
Thus we need to give upper bounds on $| \TT \1_{F_n} |_q$. We distinguish the case $k=1$ and $k \geq 2$. Let us start with $k \geq 2$.

\begin{lem}\label{lem:concentr}
Let $k\geq 2$, $R > 0$ and $F \subset \R$ such that $|F| = \delta$ and $d(0,F) \geq R$. Then
\begin{equation*}
| \TT \1_F |_q \leq 2 \delta R^\frac{d(k-d)}{d+1}.
\end{equation*}
\end{lem}

\begin{proof}
We start by a pointwise estimate: let $r>0$, we want upper bounds on $\TT \1_F(r)$. 
\begin{align*}
\TT \1_F (r) & = \disp{ \int_{u \geq \max(r,R)} \1_F (u) (u^2-r^2)^{k/2-1} u du } \\
     & = \disp{ \int_{u \geq \max(r,R)} \1_F (u) u^{d-1} (u^2-r^2)^{k/2-1} u^{2-d} du }\\
     & \leq \max(r,R)^{k-d} \delta .
\end{align*}
Here we used $k\geq 2$. Thus
\begin{align*}
| \TT \1_F |_q^q  & \leq \disp{ \int_{r\leq R} [R^{k-d} \delta]^q r^{d-k-1}dr  + \int_{r\geq R} (r^{k-d} \delta)^q r^{d-k-1}dr}\\
    & \leq 2 \delta^q R^{d(k-d)}.
\end{align*}
This immediately leads to
\begin{equation*}
| \TT \1_F |_q \leq 2 \delta R^\frac{d(k-d)}{d+1} = 2 \delta R^{-d/p}.
\end{equation*}
\end{proof}

A similar lemma for $k=1$ is the following:

\begin{lem}\label{lem:concentr1}
Let $k=1$, $R \geq 1$, $F$ a measurable set with $|F| \sim 1$ such that $d(0,F) \geq R$. Then
\begin{equation*}
| \TT \1_F |_q \lesssim R^{-1/q}.
\end{equation*}
\end{lem}

The proof is essentially in three steps: first we prove it when $F$ is an interval, then when $F$ is a countable union of intervals, and at last when $F$ is any measurable set.

\begin{proof} Let $\rho \geq R$ and $F$ be the set $(\rho,\rho+\delta)$. Then
\begin{equation*}
\delta \rho^{d-1} \sim |F| \sim 1 . 
\end{equation*}
Moreover,
\begin{equation*}
\TT \1_F(r) \leq \sqrt{(\rho+\delta)^2 - \rho^2} \sim \sqrt{\delta \rho}.
\end{equation*}
Let us look at $|\TT \1_F |_q^q$; choose $1 \geq \epsi \geq \delta$. We can cut the integral into two parts: 
\begin{align*}
| \TT \1_F |_q^q & = \disp{ \int_0^{\rho-\epsi} |\TT \1_F|^q + \int_{\rho-\epsi}^{\rho+\delta} |\TT \1_F|^q } \\
                   & \lesssim \disp{ \left(\int_0^{\rho-\epsi} |\TT \1_F|^q \right) + (\delta \rho)^{q/2} \rho^{d-2} \epsi }.
\end{align*}
Now we give estimates on
\begin{equation*}
\int_0^{\rho-\epsi} |\TT \1_F|^q .
\end{equation*}
On $(0,\rho-\epsi)$, 
\begin{align*}
\TT \1_F (r) & \leq  \TT\1_F(\rho-\epsi) \\
             & \leq \sqrt{(\rho+\delta)^2 - (\rho-\epsi)^2} - \sqrt{\rho^2 - (\rho-\epsi)^2} \\
             & \leq \sqrt{ 2 \rho \delta + \delta^2 + 2 \rho \epsi - \epsi^2 } - \sqrt{ 2 \rho \epsi - \epsi^2 } \\
             & \leq ( 2 \rho \delta + \delta^2 - \epsi^2 + \epsi^2 ) \cdot \dfrac{1}{2 \sqrt{2 \rho^2 \epsi - \epsi^2}}.
\end{align*}
Since $\delta \leq \epsi \leq 1 \leq \rho$, $\rho \epsi \geq \epsi^2$ and $\rho \delta \geq \delta^2$. As a consequence,
\begin{equation*}
\TT \1_F (r)  \lesssim \frac{\rho \delta}{\sqrt{\rho \epsi}}.
\end{equation*}
Thus
\begin{align*}
| \TT \1_F |_q^q & \lesssim (\delta \rho)^{q/2} \rho^{d-2} \epsi + \left(\sqrt{\rho\epsi} \dfrac{\delta}{\epsi} \right)^q \rho^{d-1} \\
& \lesssim  (\rho^{2-d})^{q/2-1} \epsi + \left(\sqrt{\rho} \dfrac{\rho^{1-d}}{\sqrt{\epsi}} \right)^q \rho^{d-1}.    
\end{align*}
Now it is time to precise the value of $\epsi$. We recall that $q \geq 2$. Thus
\begin{equation*}
| \TT \1_F |_q^q \lesssim  \epsi + \left( \rho^{3/2-d}\epsi^{-1/2} \right)^q \rho^{d-1}.
\end{equation*}
Now let us chose $\epsi \sim 1/\rho$. This leads to
\begin{equation*}
| \TT \1_F |_q^q \lesssim  1/\rho + \rho^{(1-d)(q-1)} \lesssim 1/R.
\end{equation*}
This proves the first step: the lemma is true when $F$ is an interval.\\

Let us consider now a set $F=E \cup I$ where $I = (a,b)$ is an interval, $a > \sup E$ and $\Delta := d(E,I) > 0$. Let us call now $F_\Delta = E \cup I_\Delta$ where $I_\Delta$ is an interval with $\inf I_\Delta = a- \Delta$ and $|I| = |I_\Delta|$ (we simply concentrate $F$ by sticking its two parts to the closest one from $0$). We want to show that
\begin{equation}\label{eq:eq32}
| \TT \1_{F_\Delta} |_q^q \geq | \TT \1_F |_q^q .
\end{equation} 
This is essentially an inverse concentration result. Indeed, it expresses the idea that the nearer from $0$ a function is, the bigger its $1$-plane transform is. 

To prove \eqref{eq:eq32}, developing both sides, it would be sufficient to prove that for all $1 \leq m \leq q-1$ we have 
\begin{equation*}
\lr{( \TT \1_E)^{q-m} , (\TT \1_{I_\Delta})^m - (\TT \1_{I})^m} \geq 0  .
\end{equation*}
Since $\sup \supp(\TT \1_E) \leq \sup E$, we just have to prove that for all $r \in E$,
\begin{equation*}
\TT \1_{I_\Delta} (r) - \TT \1_{I} (r) \geq 0.
\end{equation*}
But if $r\leq a$ then with the change of variable $u^2 = v$,
\begin{equation}\label{eq:special}
\TT \1_{I_\Delta} (r) - \TT \1_{I} (r) = \int_{(a-\Delta)^2}^{a^2} \dfrac{dv}{2 \sqrt{v-r^2}} - \int_{(b-\Delta')^2}^{b^2} \dfrac{dv}{2 \sqrt{v-r^2}}.
\end{equation}
Here $\Delta'$ is such that $I_\Delta = (a-\Delta, b-\Delta')$. Because of $|I_\Delta| = |I|$ this implies $\Delta' \leq \Delta$. The integrated function in \eqref{eq:special} is nonincreasing (in $u$) and because of $\Delta' \leq \Delta$ we have $\TT \1_{I_\Delta} (r) \geq \TT \1_{I} (r)$ for $r \in \supp (\TT \1_E)$.\\

Now let us consider a countable union of disjoint intervals,
\begin{equation*}
F = \displaystyle{ \bigcup_{m \geq 0} I_m }.
\end{equation*}
By the process described above, we can simply stick an arbitrary large, finite number of the intervals $I_m$ to the closest one from $0$. This increases the $L^q$-norm of the $1$-plane transform. To pass from a finite number to an infinite number of $I_m$ we simply notice that $|I| < \infty$. Thus the lemma is true again.\\

Finally, assume that $F$ is just a measurable set. Then by definition,
\begin{equation*}|F| = \inf \left\{ \sum_{m=1}^\infty |I_m|, I_m \text{ open interval with } F \subset \bigcup_{m\geq 0} I_m \right\}.
\end{equation*}
Approaching $F$ with a recovering of intervals such that $|\bigcup_{m \geq 0} I_m - F| \ll 1 $ shows that the lemma remains true for any measurable set.

\end{proof}

Now using Lemma \ref{lem:concentr} we get
\begin{equation}\label{eq:345}
| \TT h_n |_q \geq | \TT g_n |_q - 2 \delta_n R^{-d/p}.
\end{equation}
Let us recall that $B$ denotes the best constant in \eqref{eq:map2}. Equation \eqref{eq:345}  leads to 
\begin{equation*}
B\geq \dfrac{| \TT h_n |_q}{ | h_n |_p} \geq \dfrac{| \TT g_n |_q - 2 \delta_n R^{-d/p}}{(1-\delta_n)^\frac{1}{p}}.
\end{equation*}
Let us call  $\liminf_{n \rightarrow \infty} \delta_n(R) = l(R)$. Then $l(R) \leq 1$ and making $n \rightarrow \infty$,
\begin{equation*}
B \geq \dfrac{B- 2 R^{-d/p}}{(1-l(R))^{1/p}}.
\end{equation*}
This forces $\limsup_{R \rightarrow \infty} l(R) =0$. In particular, there exist $R_0$ and a subsequence of $\EE_n$, still called $\EE_n$, such that
\begin{equation*}
|\EE_n - [0,R_0]| \leq \frac{e}{2} .
\end{equation*}
Thus denoting $E_n = \EE_n \cap [0,R_0]$, Lemma \ref{lem:fund} is proved. The proof in the case $k=1$ is similar, using Lemma \ref{lem:concentr1}.
\end{proof}

A simple consequence is the following:

\begin{cor}
Let $f_n$ be an extremizing sequence. Then modulo the group of dilations, $f_n$ admits a subsequence that converges weakly to a non-zero function in $L^p$.
\end{cor}

\begin{proof}
Let us consider $g_n$, $R_0$, $E_n$ given by Lemma \ref{lem:fund}. $g_n$ being bounded in $L^p$, it admits a subsequence that converges weakly to a function $g$. And $g \neq 0$:
\begin{equation*}
\disp{ \int_0^{R_0} g = \lim_{n \rightarrow \infty} \int_0^{R_0} g_n \geq \lim_{n \rightarrow \infty} \int_0^{R_0} \1_{E_n} \sim 1}.
\end{equation*}
\end{proof}

Another consequence is the following:

\begin{cor}
The sequence $g_n$ introduced in Lemma \ref{lem:fund} admits a subsequence that is uniformly $L^p$-integrable, that is
\begin{equation*}
\disp{\lim_{R \rightarrow \infty} \int_{\{ g_n > R \}} |g_n|^p  =0},
\end{equation*}
uniformly in $n$.
\end{cor}

\begin{proof} We recall that $g_n$ was constructed so that the sequence $g_n^\ast$ converges (modulo subsequence) to in $L^p$ -- see the beginning of the proof of Lemma \ref{lem:fund}. Thus $g_n^\ast$ admits a subsequence that is uniformly $L^p$-integrable, and so does $g_n$, as
\begin{equation*}
\disp{ \int_{\{ g_n>R \}} |g_n|^p = \int_{\{ g_n^\ast>R \}} |g_n^\ast|^p }.
\end{equation*}
\end{proof}

This approach can be generalized even in the non-radial case, as long as our sequence of extremizing functions does not converge weakly to $0$. This is a good starting point if one wants to generalize Christ's theorem in \cite{Christ3} for the Radon transform ($k=d-1$). What is missing, however, is a complete quasiextremal theory for the inequality \eqref{eq:mapR} such as Theorem $1.2$ in \cite{Christ2} for $k=d-1$.\\

From now we will consider $g_n$ instead of $f_n$, and we will just assume the three following property, that are actually consequences of the above lemma and corollaries:
\begin{itemize}
\item The sequence $g_n$ converges weakly to a non-zero function $g \in L^p$.
\item There exist some sets $E_n$ of measure $|E_n| \sim 1$ such that $g_n \geq 1_{E_n}$ and $E_n \subset [0,R_0]$.
\item The sequence $g_n$ is uniformly $L^p$-integrable.
\end{itemize}

\section{Weak interaction}\label{sec:3}
Let us recall the following famous lemma from \cite{lions}:

\begin{lem}\label{lem:pl}
Let $\phi_n$ a sequence of nonnegative functions in $L^1(\R^d)$, such that $| \phi_n |_1=\lambda$. Then there exists a subsequence of $\phi_n$, still noted $\phi_n$, such that one of the following is satisfied:
\begin{enumerate}
\item (tightness) There exists $y_n \in \R^d$, such that for all $\epsi >0$, there exists $R >0$, for all $n$,
\begin{equation*}
\disp{ \int_{B(y_n,R)} \phi_n \geq \lambda-\epsi }.
\end{equation*}
\item (vanishing) For all $R$, 
\begin{equation*}
\disp{ \lim_{n \rightarrow \infty} \sup_{y \in \R^d} \int_{B(y,R)} \phi_n =0 }.
\end{equation*}
\item (dichotomy) There exist $0<\alpha<\lambda$ and two sequences $\phi_n^1, \phi_n^2$ of $L^1$-functions with compact support such that $\phi_n \geq \phi_n^1, \phi_n^2 \geq 0$ and
\begin{equation*} 
| \phi_n - \phi_n^1-\phi_n^2 |_1 \rightarrow 0,  \ \ | \phi_n^1 |_1 \rightarrow \alpha , \ \ | \phi_n^2 |_1 \rightarrow \lambda-\alpha,
\end{equation*}
\begin{equation*}
d( \supp(\phi_n^1) , \supp(\phi_n^2)) \rightarrow \infty .
\end{equation*}
\end{enumerate}
\end{lem}

Let us consider the sequence $\phi_n = |g_n|^p$. $\phi_n \in L^1(\R, r^{d-1}dr)$ and $|\phi_n |_1 = 1$, thus $\phi_n$ satisfies one of the three consequences above. We can see, as a consequence of the previous section, that $\phi_n$ does not satisfy $(2)$. Our purpose here is to prove that $(3)$ cannot occur either. \\

We first state a refinement of the dichotomy condition:

\begin{lem}\label{lem:supp}
Let $\phi_n = |g_n|^p$. Assume that the dichotomy condition $(3)$ hold. Then up to a subsequence, there exist $0<\alpha<1$, $R'_0>0$, and two sequences, $\varphi_n^1, \varphi_n^2$, of $L^1$-functions with compact support such that $\phi_n \geq \varphi_n^1, \varphi_n^2 \geq 0$ and
\begin{equation*} 
| \phi_n - \varphi_n^1-\varphi_n^2 |_1 \rightarrow 0,  \ \ | \varphi_n^1 |_1 \rightarrow \alpha , \ \ | \varphi_n^2 |_1 \rightarrow 1-\alpha,
\end{equation*}
satisfying the additional support condition:
\begin{equation}\label{eq:support}
\supp(\varphi_n^1) \subset [0,R'_0], \ \ \supp(\varphi_n^1) \subset [R_n, \infty)
\end{equation}
with $R_n \rightarrow \infty$.
\end{lem}

This is a consequence of the fact that $g_n \geq \1_{E_n}$ with $|E_n|\sim 1$ and $E_n \subset [0,R_0]$, using measure theory and structure of open sets in $\R$.

\begin{proof}
Let us recall that there exist some sets $E_n \subset [0,R_0]$ with measure of  order $1$ such that $g_n \geq \1_{E_n}$. Let us chose $\epsi \leq |E_n|/2$, and $\phi_n^1$, $\phi_n^2$ associated by $(iii)$ to this choice of $\epsi$. Then $\phi_n^1$ or $\phi_n^2$ must have some weight inside $[0,R_0]$, let's say $\phi_n^1$. The support separation property -- $d(\supp(\phi_n^1), \supp(\phi_n^2)) \rightarrow \infty$ -- insures that $\supp(\phi_n^2) \cap [0,R_0] = \emptyset$ for $n$ large enough.\\

Let us call $K_n := \supp(\phi_n^1)$. There exist some open sets $U_n$, containing $K_n$, with $|U_n - K_n| \ll 1$. The structure of open sets in $[0,\infty)$ allows us to write a decomposition
\begin{equation*}
U_n = \bigcup_{i \geq 0} U_n^i
\end{equation*}
with $U_n^i$ open, disjoint intervals, ordered by $\sup U_n^i \leq \inf U_n^{i+1}$, possibly enlarging $U_n$ by an arbitrary small open set containing $0$. Let us now call
\begin{equation*}
d_n^i = d([0,R_0], U_n^i).
\end{equation*}
For all $n$, the sequence $(d_n^i)_{i \geq 0}$ is increasing, and $d_n^0 = 0$. Using Cantor's diagonal argument, we can assume that for all $i$, $d_n^i \rightarrow d^i \in [0,\infty]$ as $n \rightarrow \infty$. Let us call $i_0$ the smallest integer such that $d^{i_0} = \infty$. Note that since $d_n^0 =0$, because of the existence of $E_n$, we have $i_0 \geq 1$. If $i_0 = \infty$, then the lemma is proved, we do not have to change $\phi_n^1, \phi_n^2$. If $i_0 < \infty$, let us call
\begin{equation*}
\disp{U'_n := \bigcup_{i=0}^{i_0-1} U_n^i}, \ \ \rho_0 := \sup \{\sup(U'_n), \ n \geq 0 \} < \infty .
\end{equation*}
We can change the sequences $\phi_n^1$, $\phi_n^2$ to
\begin{equation*}
\varphi_n^1 = \phi_n^1 \1_{U'_n}, \ \ \varphi_n^2 = \phi_n^1 \1_{U_n - U'_n} + \phi_n^2 .
\end{equation*}
It remains to prove that $\varphi_n^1, \varphi_n^2$ satisfy the conclusions of the lemma. Up to a subsequence, we can assume that $| \varphi_n^1 |_1$ converges to some $1 \geq\az'\geq 0$. Moreover, we have $0 \leq \varphi_n^1 \leq \phi_n^1$ and thus $\az' < 1$. Since $\varphi_n^1 + \varphi_n^2 = \phi_n^1+\phi_n^2$, $| \phi_n - \varphi_n^1 - \varphi_n^2 |_1 \rightarrow 0$ and the additional support condition \eqref{eq:support} is satisfied -- with $R_n=d_n^{i_0}$. At last, $\az'>0$: indeed,
\begin{equation*}
| \phi_n - \varphi_n^1 - \varphi_n^2 |_1 \rightarrow 0 \Rightarrow \disp{ \int_0^{\rho_0} | \phi_n - \varphi_n^1 - \varphi_n^2| = \int_0^{\rho_0} | \phi_n - \varphi_n^1| \rightarrow 0 } .
\end{equation*}
But $\phi_n \geq \1_{E_n}$ with $|E_n| \sim 1$ which implies $| \varphi_n^1 |_1 \sim 1$, thus $\az' >0$.
\end{proof}
We will not change the notations for purpose of simplicity: we will assume without loss of generality $\az' = \az$ and $\rho_0 = R_0$.\\

We define $g_n^1 = |\varphi_n^1|^{1/p}$ and $g_n^2 = |\varphi_n^2|^{1/p}$. We are ready to prove our claim: dichotomy cannot occur. Let us be guided by the interpretation of Lemma \ref{lem:concentr} (or Lemma \ref{lem:concentr1} in the case $k=1$): radial characteristic functions of sets that are far away from $0$ cannot have a large $k$-plane transform. Here, $\gtd$ \textit{is} far away from $0$. However, we have no assumption on the support of the size of the support of $\gtd$ and this is a real difficulty: Lemma \ref{lem:concentr}, \ref{lem:concentr1} cannot be directly applied and must be generalized to \textit{any} function. Let us give an idea of the difficulty: roughly, we can write
\begin{equation*}
\disp{\gtd = \sum_{m \in \Z} 2^m \1_{E_m}}
\end{equation*}
with $E_m = \{ r \geq 0 \text{ with } \gtd(r) \sim 2^m \}$. Each $E_m$ must be far from $0$. As a consequence, 
\begin{equation*}
| \TT \gtd |_q^q = \disp{ \sum_{m \in \Z} 2^{mq} | \TT \1_{E_m} |_q^q  + \sum_{m_1, ..., m_q} 2^{(m_1+...+m_q)q} \int \TT \1_{E_{m_1}} \cdot ... \cdot \TT \1_{E_{m_q}}  } .
\end{equation*}
The first sum is called principal sum (and its components principal terms) while the second is called interaction sum (and its components interaction terms); this is a notion we will encounter again below. It is not be too hard to prove that the principal sum tends to $0$ as $n$ tends to infinity: it follows from Lemma \ref{lem:concentr}. However, there are too many terms in the interaction sum. Lemma \ref{lem:concentr} would not be of any help. Approaches of this type have however proved to work modulo extra work, see \cite{Christ2}, Theorem 1.6. We chose to take here another path.\\

Let us get inspired by \cite{lions}. We introduce the quantities
\begin{equation*}
S_\alpha := \sup \{ | \TT f |_q^q, | f |_p^p = \alpha \}.
\end{equation*}
Then the following convexity inequality is satisfied: for all $0 < \alpha < 1$,
\begin{equation}\label{eq:strict}
S_1 > S_\alpha + S_{1 - \alpha}.
\end{equation}
This comes from the fact that $S_\az = \az^\frac{q}{p} S_1$ and the convexity inequality
\begin{equation*}
1 > \az^\frac{q}{p} + (1-\az)^\frac{q}{p},
\end{equation*}
for $q>p$.\\

Assume now that dichotomy can occur. The two sequence $\gtu $, $\gtd $, satisfy then
\begin{enumerate}
\item[$(i)$] $\supp(\gtu) \subset [0,R_0]$, $d(\supp(\gtu), \supp(\gtd)) \rightarrow \infty$.
\item[$(ii)$] $| \gtu |_p^p \rightarrow \az \in (0,1)$, $| \gtd |_p^p \rightarrow 1-\az$ as $n \rightarrow \infty$.
\item[$(iii)$] $| |g_n|^p - |\gtu+\gtd|^p |_1 \rightarrow 0$ as $n \rightarrow \infty$.
\item[$(iv)$] $g_n \geq \gtu, \gtd \geq 0$.
\end{enumerate}
Our purpose is to prove a contradiction with \eqref{eq:strict}. The essential idea is:

\begin{lem}
Let $g_n$, $\gtu$, $\gtd$ as above. Then
\begin{equation}\label{eq:strict1}
| \TT g_n |_q^q - | \TT \gtu |_q^q - | \TT \gtd |_q^q \rightarrow 0 .
\end{equation}
\end{lem}

It means in a way that $\TT g_n^1$ and $\TT g_n^2$ interact weakly, or are asymptotically orthogonal. This will be a contradiction with \eqref{eq:strict}. It is interesting to relate this lemma to the discussion made above. Here we do not make $\TT\gtd$ interact with itself, but with  $\TT\gtu$. Since the supports of $\gtu$ and $\gtd$ are far away from each other it is actually much easier.

\begin{proof} Let us first define $\epsi_n := g_n - \gtu-\gtd$. Then $| \epsi_n |_p \rightarrow 0$. Indeed,
\begin{align*}
| \epsi_n |_p^p & = \disp{ \int_0^\infty |g_n -\gtu- \gtd|^p }\\
     & \leq \disp{ \int_0^\infty |g_n|^p - |\gtu + \gtd|^p } \rightarrow 0,
\end{align*}
because of $(iii)$ in Lemma \ref{lem:pl}. Moreover,
\begin{align*}
| \TT g_n |_q^q & \leq | \TT \gtu + \TT \gtd |_q^q + O(| \epsi_n |_p) \\
                  & \leq \disp{|  \TT \gtu |_q^q + | \TT \gtd |_q^q + \sum_{1 \leq m \leq q-1} \binom{q}{m} \lr{ (\TT \gtu)^{q-m}, (\TT \gtd)^m } + o(1)}.
\end{align*}
To prove \eqref{eq:strict1} it is sufficient to show that for $1 \leq m \leq q-1$, $\lr{ (\TT \gtu)^{q-m}, (\TT \gtd)^m } \rightarrow 0$. Let $\epsi >0$. Since $g_n$ is uniformly $L^p$-integrable, there exists $R_\epsi$ such that 
\begin{equation*}
\disp{ \int_{\gtu \geq R_\epsi} |\gtu|^p \leq \int_{g_n \geq R_\epsi} |g_n|^p \leq \epsi},
\end{equation*}
uniformly in $n$. Thus, in a way, it is sufficient to prove \eqref{eq:strict1} for $\gtu \leq R_\epsi$. Indeed,
\begin{align*}
| \lr{ (\TT \gtu)^{q-m}, & (\TT \gtd)^m } - \lr{ (\TT 1_{\{\gtu \leq R_\epsi\}}\gtu)^{q-m}, (\TT \gtd)^m } |\\ 
   &= | \lr{ (\TT (\1_{\{\gtu \leq R_\epsi\}}\gtu + \1_{\{\gtu > R_\epsi\}}\gtu  ))^{q-m} - (\TT 1_{\{\gtu \leq R_\epsi\}}\gtu)^{q-m}, (\TT \gtd)^m }| \\ 
   & \leq  \disp{ \sum_{1\leq i \leq m-q} \binom{m-q}{i} |  \lr{ (\TT (\1_{\{\gtu \leq R_\epsi\}}\gtu)^{q-m-i} (\TT \1_{\{\gtu > R_\epsi \}}\gtu  ))^i, (\TT \gtd)^m }  |}.
\end{align*}
Let us treat independently the quantities $\lr{ (\TT (\1_{\{\gtu \leq R_\epsi\}}\gtu)^{q-m-i} (\TT \1_{\{\gtu > R_\epsi\}}\gtu )^i, (\TT \gtd)^m }  $ that appeared just above. Using H\"older's inequality with
\begin{equation*}
1 = \dfrac{q-m-i}{q}+ \dfrac{i}{q} + \dfrac{m}{q},
\end{equation*}
we get
\begin{align*}
|\lr{ (\TT (\1_{\{\gtu \leq R_\epsi\}}& \gtu)^{q-m-i} (\TT \1_{\{\gtu > R_\epsi\}}\gtu  ))^i,  (\TT \gtd)^m }  | \\
&  \leq | (\TT \1_{\{\gtu \leq R_\epsi\}}\gtu |_q^{q-m-i} \cdot | \TT \1_{\{\gtu > R_\epsi\}}\gtu   |_q^i \cdot | \TT g_n^2 |_q^m \\
   & \leq A_\TT^q \cdot | g_n^1 |_p^{q-m-i} \cdot | \1_{\{\gtu > R_\epsi\}}\gtu |_p^i \cdot | g_n^2 |_p^m \\
   & \lesssim \epsi^\frac{i}{p}.
\end{align*}
Coming back to our initial point,
\begin{equation*}
| \lr{ (\TT \gtu)^{q-m}, (\TT \gtd)^m } - \lr{ (\TT 1_{\{\gtu \leq R_\epsi\}}\gtu)^{q-m}, (\TT \gtd)^m } | \lesssim \epsi^{1/p}
\end{equation*}
and this bound is uniform in $n$. Thus we can assume without loss of generality $g_n^1 \leq R_\epsi$. This leads to
\begin{align*}
\lr{ (\TT \gtu)^{q-m}, (\TT \gtd)^m } & \leq R_\epsi^{q-m} \lr{ (\TT \1_{[0,R_0]})^{q-m}, (\TT \gtd)^m } \\ 
     & \lesssim R_\epsi^{q-m} R_0^{(q-m)k} \lr{ \1_{[0,R_0]}, (\TT \gtd)^m  }.
\end{align*}

\begin{lem}\label{lem:weak}
Let $\psi \in L^p$ and $R \geq 1$ such that $\delta := d(\supp(\psi), [0,R]) \geq R$. Then for all $1 \leq m \leq q-1$,
\begin{equation*}
\lr{\1_{[0,R]}, (\TT \psi)^m } \lesssim \dfrac{R^{d-k}}{(R+\delta)^\frac{m}{p'}} | \psi |_p^m .
\end{equation*}
\end{lem}

\begin{proof}
It is only some simple calculation. Indeed,
\begin{align*}
 & \lr{\1_{[0,R]}, (\TT \psi)^m } = \disp{ \int_0^R \left(\int_{u\geq r} \psi(u) (u^2-r^2)^{k/2-1} udu \right)^m r^{d-k-1}dr }\\
    & = \disp{    \int_0^R \left(\int_{u_1\geq R+\delta} \psi(u_1)k(u_1,r)u_1du_1 \right) \cdot ... \cdot \left(\int_{u_m\geq R+\delta} \psi(u_m)k(u_m,r)u_mdu_m \right)  r^{d-k-1}dr } \\
    & = \disp{\int_{u_1\geq R+\delta} ... \int_{u_m\geq R+\delta} \psi(u_1) u_1 du_1 ... \psi(u_m) u_m du_m  \int_0^R k(u_1,r) ... k(u_m,r)  r^{d-k-1} dr },\\
\end{align*}    
where we denote  by $k(u,r)$ the kernel $(u^2-r^2)^{k/2-1}$. For $k \geq 2$,
\begin{equation*}
\int_0^R k(u_1,r) ... k(u_m,r)  r^{d-k-1} dr \leq u_1^{k-2} ... u_m^{k-2} R^{d-k}.
\end{equation*}
For $k=1$,
\begin{align*}
\int_0^R k(u_1,r) ... k(u_m,r)  r^{d-k-1} dr & \leq \int_0^R \dfrac{1}{\sqrt{u_1^2-r^2}} ... \dfrac{1}{\sqrt{u_m^2-r^2}}  r^{d-k-1} dr \\
    & \leq \dfrac{1}{u_1} ... \dfrac{1}{u_m} \int_0^R \left(1-\dfrac{r^2}{u_1^2} \right)^{-1/2} ... \left( 1-\dfrac{r^2}{u_m^2} \right)^{-1/2}  r^{d-k-1} dr \\
  &  \lesssim  \dfrac{1}{u_1} ... \dfrac{1}{u_m} R^{d-k}
\end{align*}
since $u_i \geq R+\delta \geq 2R \geq 2r$.

As a consequence,
\begin{align*}
\lr{\1_{[0,R]}, (\TT \psi)^m }        & \lesssim \int_{u_1\geq R+\delta} ... \int_{u_m\geq R+\delta} \psi(u_1) u_1 du_1 ... \psi(u_m) u_m du_m u_1^{k-2} ... u_m^{k-2} R^{d-k}\\
    & = \disp{ R^{d-k}  \left(\int_{u \geq R+ \delta} \psi(u) u^{k-1} du  \right)^m}.
\end{align*}
Noting that
\begin{equation*}
{u^{k-1} = u^{-\frac{2}{p'}} u^\frac{d-1}{p},}
\end{equation*}
the H\"older inequality leads to:
\begin{equation*}
\int_{u \geq R+ \delta} \psi(u) u^{k-1} du \lesssim \left(\int_{u \geq R+\delta} \psi(u)^p u^{d-1} du \right)^{1/p}  \left(\int_{u \geq R+\delta} u^{-2} du \right)^\frac{1}{p'} = \dfrac{1}{(R+\delta)^\frac{1}{p'}} | \psi |_p. 
\end{equation*}
Thus, finally,
\begin{equation*}
\lr{\1_{[0,R]}, (\TT \psi)^m} \lesssim \dfrac{R^{d-k}}{(R+\delta)^\frac{m}{p'}} | \psi |_p^m .
\end{equation*}
\end{proof}

For $n$ large enough, a direct application of Lemma \ref{lem:weak} leads to
\begin{equation*}
\lr{ \1_{[0,R_0]}, (\TT \gtd)^m } \lesssim \dfrac{R_0^{d-k}}{(R_0+d(R_0,\supp(\gtd))^\frac{m}{p'}} | \gtd |_p^m \lesssim \dfrac{1}{d(R_0,\supp(\gtd))^\frac{m}{p'}}.
\end{equation*}
Thus we get
\begin{equation*}
\lr{ (\TT \gtu)^{q-m}, (\TT \gtd)^m } \lesssim R_\epsi^{q-m} R_0^{(q-m)k} \dfrac{R_0^{d-k}}{d(R_0,\supp(\gtd))^\frac{m}{p'}}.
\end{equation*}
Making $n \rightarrow \infty$ leads to the conclusion. The uniformity of $R_\epsi$ in $n$ is crucial in the proof.
\end{proof}

Now let us chose two sequences $\alpha_n , \beta_n \rightarrow 1$ with $| \az_n \gtu |_p^p =1$ and $ | \beta_n \gtd |_p^p = 1-\az $. Then as $n \rightarrow \infty$,
\begin{equation*}
S_\az+S_{1-\az} \leftarrow S_\az + S_{1-\az} + o(1) \geq | \TT \az_n \gtu |_q^q + | \beta_n \TT \gtd |_q^q + o(1) = | \TT g_n |_q^q \rightarrow S_1
\end{equation*} 
which is a contradiction with \eqref{eq:strict}. Thus dichotomy cannot occur.

A simple consequence is:

\begin{cor}
$g_n$ is strongly tight, that is
\begin{equation*}
\disp{ \lim_{R \rightarrow \infty} \int_R^\infty |g_n|^p = 0 },
\end{equation*}
uniformly in $n$.
\end{cor}

Indeed, since $|g_n|^p \geq \1_{E_n}$, $g_n$ cannot be vanishing. Thus $|g_n|^p$ must be tight. But since $E_n \subset [0,R_0]$ the sequence $y_n$ involved in $(i)$ in Lemma \ref{lem:pl} can be chosen to be $0$, involving a possible redefinition of $R_0$.

\section{Proof of Theorem \ref{thm:precompact2}.}\label{sec:4}

Here we prove Theorem \ref{thm:precompact2}. It is a consequence of the following Proposition:

\begin{proposition}\label{prop:converT}
$\TT g_n$ converges strongly in $L^q$.
\end{proposition}
We first start to prove that the operator $\TT$ is somehow locally compact.

\begin{lem}\label{lem:comp}
Let us consider  for $R>0$ the operator
\begin{center}
$\begin{array}{ccccc}
& \TT_R : & L^\infty([0,R]) & \longrightarrow & L^q([0,R])\\
& & f & \longmapsto & \TT \1_{[0,R]} f = \TT f. \end{array}$
\end{center}
Then $\TT_R$ is compact.
\end{lem}

We give the proof of this result in the appendix. Let us now define
\begin{equation*}
g_n^m := \1_{\{f_n \leq m\}} \1_{[0,m]} g_n,
\end{equation*}
\begin{equation}\label{eq:limi}
h^m := \lim_{n \rightarrow \infty} \TT g_n^m .
\end{equation}
The convergence in \eqref{eq:limi} occurs in $L^q$, because of the local compactness of $\TT$ -- Lemma \ref{lem:comp}. Indeed, the sequence $(g_n^m)_n$ is bounded in $L^\infty([0,m])$, thus the sequence $(\TT g_n^m)_n$ is compact in $L^q$. Moreover $0 \leq g_n^m \leq g_n^{m+1} \leq g_n$, implying $0 \leq \TT g_n^m \leq \TT g_n^{m+1} \leq \TT g_n $. It proves that the sequence $h^m$ is actually nondecreasing, nonnegative, and bounded in $L^q$. We can then apply the monotonous convergence theorem: there exists $h \in L^q$ with $h^m \rightarrow h$ strongly. We want to show now that $\TT g_n$ converges strongly to $h$.
\begin{align}
\lim_{n \rightarrow \infty} \TT g_n - h & = \lim_{n \rightarrow \infty} \lim_{m \rightarrow \infty} \TT g_n - h^m \\
 &=  \lim_{n \rightarrow \infty} \lim_{m \rightarrow \infty} \lim_{n \rightarrow \infty}  \TT g_n - \TT g_n^m \\
\label{eq:interv} &=  \lim_{n \rightarrow \infty} \lim_{m \rightarrow \infty} \TT g_n - \TT g_n^m = 0.
\end{align}
In \eqref{eq:interv}, we were allowed to change the order of the limits $n \rightarrow \infty$, $m\rightarrow \infty$ because of the uniform convergence (in $n$) of $g_n^m$ to $g_n$ as $m\rightarrow \infty$.

Now we know that $\TT$ is linear continuous from $L^p$ to $L^q$, so it is also continuous when these spaces are provided with the weak topology. It shows that  $\TT g_n \rightharpoonup \TT g$. But $\TT g_n \rightarrow h$, thus $\TT g = h$ and $\TT g_n \rightarrow \TT g$. This proves Proposition \ref{prop:converT}.\\

It implies that the sequence $g_n$ converges weakly to an extremizer, since $| g |_p \leq \liminf |g_n |_p = 1$. This also implies $| g |_p=1$. Using uniform convexity of $L^p$ this proves $g_n \rightarrow g$, which is Theorem \ref{thm:precompact2}.

\section{Preliminaries for Theorem \ref{thm:local}.}\label{sec:12}
\subsection{The manifold of extremizers}
Let $\EE$ be the set of extremizers of the inequality
\begin{equation*}
|\RR f|_q \leq A_\RR |f|_p,
\end{equation*}
that is, the set of non-zero functions defined by \eqref{eq:hg}. It is the orbit of the single extremizer $h_0$ defined by \eqref{eq:extrpa} under the action of $\Gg$

\begin{lem}\label{lem:tanR} \begin{enumerate}
\item[$(i)$] $\EE$ is a smooth manifold of dimension
\begin{equation*}
\dfrac{(d+2)(d+1)}{2}.
\end{equation*}
\item[$(ii)$] If $L$ is an affine map (non necessarily invertible) then for all $x \in \R^d$,
\begin{equation*}
\left( \left. \dfrac{d}{dt} \right|_{t=0}  h_0 \circ (\Id+tL) \right)(x) = - (k+1) \dfrac{\lr{Lx,x}}{1+|x|^2} h_0(x).
\end{equation*}
\item[$(iii)$] The tangent space $T_{h_0} \EE$ is given by
\begin{equation*}
T_{h_0} \EE = \left\{ x \mapsto \left( \az +  \dfrac{\lr{Sx+x_0,x} }{1+|x|^2} \right) h_0(x), \ \ S \in \Sym(\R^d), \ x_0 \in \R^d,  \ \az \in \R \right\}.
\end{equation*} 
\end{enumerate}
\end{lem}

\begin{proof} For $(i)$, note that
\begin{equation*}
\begin{matrix}
\phi : \R-\{0\} \times \Sym(\R^d) \times \R^d & \rightarrow & \EE \\
 (C,S,x_0)  & \mapsto  & \left( x \mapsto C h_0(Sx+x_0)\right)
\end{matrix}
\end{equation*}
is a smooth bijective map.
$(ii)$ is a simple calculation. For $(iii)$, note that every linear map can be decomposed uniquely as the sum of an antisymmetric map and a symmetric map, and that the contribution of the antisymmetric in the scalar product evaluation is $0$.
\end{proof}

\subsection{Conformal invariance of \eqref{eq:mapR}} Here we describe in what sense the inequality \eqref{eq:mapR} is conformally invariant. The base paper is \cite{Drury}, which we briefly recall here. For a point $x \in \R^d$, we define $\DD(x) \in \GG_{1,d+1}$ as the only straight line in $\R^{d+1}$ containing $(x,-1)$ and passing through the origin. If $f \in L^p(\R^d)$, we define $\Ii f$ through the implicit formula
\begin{equation*}
(\Ii f)(\DD(x)) = c_d^{-1} \lr{x}^{k+1} f(x).
\end{equation*}
Note that $h_0$ is mapped to the constant function $\1$ on $\GG_{1,d+1}$ by $\Ii$. 

Given a map $F : \GG_{1,d+1} \mapsto \R$ and $\PPP \in \GG_{k+1,d+1}$ we define
\begin{equation*}
\SSS F(\PPP) = \int_{\DD \subset \PPP} F(\DD) d\lambda_\PPP(\DD).
\end{equation*}
The space of ($1$-dimensional) lines in $\PPP$ is here again provided with the invariant probability measure $d\lambda_\PPP$. Drury's result is as follows:

\begin{theorem} (Drury) For all smooth compactly supported function $f : \R^d \rightarrow \R$,
\begin{equation*}
|\Ii f|_{L^p(\GG_{1,d+1})} = |f|_{L^p(\R^d)}
\end{equation*}
and
\begin{equation*}
c_d |\SSS f|_{L^q(\GG_{d+1,k+1})} = |\RRR f|_{L^q(\MM_{k,d})}.
\end{equation*}
\end{theorem}

Let us mention a few comments about this result. The sharp inequalities induced by the boundedness of
\begin{equation*}
\RRR : L^p(\R^d) \rightarrow  L^q(\MM_{k,d}), \  \ \SSS : L^p(\GG_{1,d+1}) \rightarrow L^q(\GG_{k+1,d+1})
\end{equation*} 
are equivalent. The great feature of $\SSS$ is its rotation invariance in the following sense. The orthogonal group $O(d+1)$ acts naturally on the Grassmanian manifolds $\GG_{1,d+1}$ and $\GG_{k,d+1}$. The operator $\SSS$ commutes with this action: if $\Omega$ is in $O(d+1)$,
\begin{equation*}
 \SSS (F \circ \Omega) = (\SSS F) \circ \Omega.
\end{equation*}
The inequality \eqref{eq:mapR} admits clearly $O(d)$ as a symmetry group. In order to prove Theorem \ref{thm:mapR}, \cite{Christ4} and \cite{Drouot} used a hidden symmetry, whose origin was unclear. It happens that this symmetry has a very simple form on the $\SSS$-side: it is the reflection across the hyperplane $\{ (x'',x_d,x_{d+1}) : x_d+x_{d+1} = 0 \}$ in $\R^{d+1}$. This clarifies its origin\footnote{Neither Pr. Christ nor the author were aware of Drury's paper while writing \cite{Christ4} and \cite{Drouot}.}.

Rotational invariance usually implies diagonalization using spherical harmonics. Here as $\SSS$ maps functions on $\GG_{1,d+1}$ to functions on $\GG_{k+1,d+1}$ it is slightly more subtle but the situation is similar. We will make a great use of this feature in Section \ref{sec:13}: it will allow us to diagonalize explicitly a selfadjoint operator. This is why in order to prove Theorems \ref{thm:global} and \ref{thm:local} we use the equivalence between $\RRR$ and $\SSS$ and prove the Theorem for the $\SSS$-inequality. In order to do that we first reformulate Theorem \ref{thm:mapR} in terms of the operator $\SSS$. The free, transitive action of $\Gg$ on extremizers of the $\RRR$-inequality induces a free, transitive action of a group $\Hh$ (which is simply the conjugated of $\Gg$ by $\Ii$) on extremizers for the $\SSS$-inequality. If $F$ is an extremizer for $S$ and $\phi \in \Hh$ we denote by $\phi \star F$ this action. This gives:

\begin{theorem} For all $F \in L^p(\GG_{1,d+1})$,
\begin{equation*}
\dfrac{|\SSS F|_{L^q(\GG_{k,d+1})}}{|F|_{L^p(\GG_{1,d+1})}} \leq A_\SSS, \ \ \ A_\SSS = \dfrac{A_\RR}{c_d} = \left( 2^{k-d} |S^k|^d |S^d| \right)^{\frac{1}{d+1}}
\end{equation*}
with equality if and only if for some symmetry $\phi \in \Hh$,
\begin{equation}\label{eq:extS}
F = \phi \star \1.
\end{equation}
\end{theorem}

We denote by $\FF$ the set of extremizers, which is, by \eqref{eq:extS}, the orbit of the constant function $\1$ under the group of symmetries $\Hh$. The function $\1$, by its simplicity, is the easiest extremizer to use. We can reformulate Lemma \ref{lem:tanR} in the following:

\begin{lem}\label{lem:tanS} \begin{enumerate}
\item[$(i)$] $\FF$ is a smooth manifold of dimension
\begin{equation*}
\dfrac{(d+2)(d+1)}{2}.
\end{equation*}
\item[$(ii)$] The tangent space $T_\1 \FF$ can be identified to the space spanned by spherical harmonics of order $0$ and $2$.
\end{enumerate}
\end{lem}

\begin{proof} For $(i)$, just note that $d\Ii_{h_0}$ is a invertible linear map from $T_\1 \EE$ to $T_{h_0} \FF$. As a consequence,
\begin{equation*}
\dim (T_\1 \FF) = \dim (T_{h_0} \EE) = \dim(\EE) = \dfrac{(d+1)(d+2)}{2}.
\end{equation*}
For $(ii)$ note that the Grassmanian $\GG_{1,d+1}$ is simply the homogeneous space $\Ss^d/\{\pm1\}$. Therefore functions on $\GG_{1,d+1}$ can be identified to even functions of the sphere. Denote $V_0$ the space of constant functions on $\Ss^d$ and $V_1$ the space spanned by second order spherical harmonics. The dimension of $V_0 \oplus V_1$ is given by (see \cite{Samko}, Lemma $1.4$)
\begin{equation*}
\dfrac{(d+1)(d+2)}{2},
\end{equation*}
As a consequence it will suffice to prove that under the identification mentioned above $T_1\FF \subset V_0 \oplus V_1$. We have
\begin{equation*}
T_\1 \FF  = d\Ii_{h_0} \left(T_{h_0}\EE\right) = \Ii(T_{h_0} \EE)
\end{equation*}
as $\Ii$ is a linear transformation. $T_{h_0} \EE$ was identified in Lemma \ref{lem:tanR}. If $\az \in \R, x_0 \in \R^d$, and $S$ is symmetric, 
\begin{equation*}
\Ii\left(  \left( \az + \dfrac{\lr{Sx+x_0, x}}{1+|x|^2} \right)h_0\right)( \DD(x)) = \az + \dfrac{\lr{Sx+x_0,x}}{1+|x|^2}. 
\end{equation*}
The case $S=0, x_0=0$ leads to functions in $V_0$. Let us work on the case $\az=0$ now. We must prove that 
\begin{equation*}
\DD(x) \mapsto \dfrac{\lr{Sx+x_0,x}}{1+|x|^2} \in V_1.
\end{equation*}
The natural action of an element $\Omega \in O(d) \times \{1\} \subset O(d+1)$ leaves the RHS invariant and changes $S$ to $\Omega^{-1}S\Omega$, $x_0$ to $\Omega^{-1} x_0$. Using the spectral theorem it then suffices to prove:
\begin{equation*}
\DD(x) \mapsto \dfrac{x_1^2}{1+|x|^2} \in V_1 \ \text{and} \ \DD(x) \mapsto \dfrac{x_1}{1+|x|^2} \in V_1.
\end{equation*}
In spherical coordinates $(\te,\varphi_1,..., \varphi_{d-1}) \in \Ss^d$,
\begin{equation*}
\sin(\te)\cos(\varphi_1)  = \dfrac{x_1}{\lr{x}}, \ \ \cos(\te) = -\dfrac{1}{\lr{x}}.
\end{equation*}
As a consequence, using the identification $\GG_{1,d+1} \equiv \Ss^d/\{\pm 1\}$, it is sufficient to check that  
\begin{equation*}
(\te,\varphi_1, ..., \varphi_{d-1}) \mapsto \sin^2(\te)\cos(\varphi_1)^2,  \ \ (\te,\varphi_1, ..., \varphi_{d-1}) \mapsto \cos(\te)\sin(\te)\cos(\varphi_1)
\end{equation*}
are spherical harmonics of order $2$ for $\Delta_{\Ss^d}$ -- which is true.
\end{proof}

Because of all that, Theorem \ref{thm:local} is equivalent to
\begin{theorem}\label{thm:local2}
 There exists a constant $\epsi_0 >0$ such that the following is valid. For all $1 \leq k \leq d-1$,
\begin{equation*}
\dfrac{|\SSS F|_q}{|F|_p}\leq A_\SSS \left( 1 - \lambda_\SSS \left( \dfrac{d^\star(F,\FF)}{|F|_p} \right)^2 + o\left(\left( \dfrac{d^\star(F,\FF)}{|F|_p} \right)^{2+\epsi_0}\right)\right),
\end{equation*}
where the constant $\lambda_\SSS$ is defined by
\begin{equation*}
\lambda_\SSS = \dfrac{d-k}{2(d+1)}-\dfrac{1}{2d} \left(\dfrac{3}{d+2}\right)^2.
\end{equation*}
\end{theorem}

The quantity $d^\star(F,\FF)$ will be defined below. The constant $\lambda_\SSS$ is optimal.

\section{Bounds on $Q_\1$.} \label{sec:13}
Recall $\HH  = L^2(\GG_{1,d+1})$. As we saw the tangent space $T_\1 \FF$ is spanned by some spherical harmonics, which are smooth bounded functions on the sphere. Therefore it can be seen as a subspace of $\HH$. Define the orthogonal space at $\1$:
\begin{equation*}
T_\1 \FF ^\perp = \left\{ G \in \HH, \int_{\GG_{1,d+1}} G H = 0 \ \text{for all} \ H \in T_\1 \FF \right\} 
\end{equation*}
We want to study, for $G \in T_\1 \FF ^\perp$
\begin{equation}\label{eq:hZ}
t \mapsto \dfrac{|\SSS(\1+tG)|_q}{A_\SSS|\1+tG|_p}
\end{equation}
as $t \rightarrow 0$. A natural strategy is to study the natural quadratic form appearing while doing a formal Taylor development. Unfortunately this development can be only formal as the function defined by \eqref{eq:hZ} is not $C^2$ in general. We will fix this technical difficulty in the next session.\\

This quadratic form is formally defined by
\begin{equation*}
Q_{\1}(G) = \dfrac{q-1}{2 |\SSS \1|_q^q} \int_{\GG_{k+1,d+1}} (\SSS G)^2 (\SSS \1)^{q-2} - \dfrac{p-1}{2} \int_{\GG_{1,d+1}} G^2,
\end{equation*}
see \cite{Christ5}, Section $15$. Here, $\SSS \1$ is a constant, whose value is $|\SSS \1|_q = A_\SSS$. Therefore
\begin{equation*}
Q_{\1}(G) = \dfrac{q-1}{2 A_\SSS^2} \int_{\GG_{k+1,d+1}} (\SSS G)^2 - \dfrac{p-1}{2} \int_{\GG_{1,d+1}} G^2.
\end{equation*}

In order to prove \ref{thm:local2} we need at least to be able that the quadratic form $Q_\1$ is negative on $T_\1 \FF ^\perp$ -- it clearly cannot be negative on the whole $\HH$ as $\FF$ has some path connected subsets. As noted first by \cite{BiaEgn} for the Sobolev inequality this can be done using the minmax principle to the operator $\LL =\SSS^* \SSS$ on $\HH$. Thus we need to perform a spectral analysis of this operator.

\begin{lem} The operator $\LL$ is bounded and selfadjoint from $\HH$ to $\HH$. 
\end{lem}

\begin{proof}
$\LL$ naturally maps $L^p(\GG_{1,d+1}) \rightarrow L^{p'} (\GG_{1,d+1})$. If $p \leq 2$, then $\HH \hookrightarrow L^p(\GG_{1,d+1})$ and $L^{p'}(\GG_{1,d+1}) \hookrightarrow \HH$. As a consequence $\LL$ maps $\HH$ to $\HH$. 

If $p \geq 2$, then as $2 \leq p \leq q$ and $\SSS$ maps $L^1(\GG_{1,d+1})$ to $L^1(\GG_{k+1,d+1})$ we can use interpolation theory to get $\SSS : \HH \rightarrow L^r(\GG_{k+1,d+1})$ for some $r \geq 2$. Now $\SSS^* : L^{r'}(\GG_{k+1,d+1}) \rightarrow \HH$. Thus again $\LL$ maps $\HH$ to $\HH$.
\end{proof}

Using rotational invariance of $\SSS$, $\LL$ is also rotation-invariant and acts from function on $\GG_{1,d+1}$ to functions on $\GG_{1,d+1}$. The space $\GG_{1,d+1}$ is a homogeneous space (with isometry group $O(d+1)$). Therefore $\LL$ commutes with the Laplace-Beltrami operator $\Delta_{\GG_{1,d+1}}$. It follows that these two operators have the same eigenfunctions and so we can explicitly diagonalize $\LL$ in terms of $\Delta_{\GG_{1,d+1}}$. But the Grassmanian $\GG_{1,d+1}$ is simply $\Ss^d/\{ \pm 1 \}$. It follows that eigenfunctions of $\Delta_{\GG_{1,d+1}}$ are simply even spherical harmonics. The following Lemma follows:

\begin{lem}\label{lem:boundQ} \begin{enumerate}
\item[$(i)$] The operator $\Delta_{\GG_{1,d+1}}$ has eigenvalues $-2\ell ( 2\ell + d-1)$, $\ell = 0, 1, ...$. The sum of the first and the second eigenspaces is $T_\1 \FF$.
\item[$(ii)$] The spectrum of the operator $\LL$ is given by an explicit decreasing sequence $\lambda_\ell$.
\item[$(iii)$] If $\lambda_\SSS$ is the (positive) number defined by
\begin{equation*}
\lambda_\SSS = \dfrac{d-k}{2(d+1)}-\dfrac{1}{2d} \left(\dfrac{3}{d+2}\right)^2  
\end{equation*}
then for all $G \in T_\1 \FF^\perp$, 
\begin{equation*}
Q_\1(G) \leq -\lambda_\SSS |G|_\HH^2.
\end{equation*}
\end{enumerate}
\end{lem}

\begin{proof} 
$(i)$ follows from the fact that functions on $\GG_{1,d+1}$ are even functions on the sphere and that the Laplace-Beltrami operator on $\Ss^d$ restricted to even functions has spectrum $2\ell(2\ell+d-1)$ -- see \cite{Samko}, Lemma $1.11$. 

\item[$(ii)$] Use again the identification $\GG_{1,d+1} \equiv \Ss^d/\{ \pm 1 \}$ and \cite{Samko}, Theorem $6.4$. We are concerned only by even functions, thus we restrict our attention to even order eigenspaces. The  eigenvalues of $\LL$, denoted by $\lambda_\ell$ are given by the equation
\begin{equation*}
\LL Y_{2\ell}^0 = \lambda_\ell Y_{2\ell}^0. 
\end{equation*}
Here $Y_{2\ell}^0$ is a zonal spherical harmonic of order $\ell$ (that is, the normalized spherical harmonic of order $2\ell$ on $\Ss^d$ that is invariant under the action of $O(d) \times \{ 1 \}$).\\

In order to compute $\lambda_\ell$ we note that $Y_{2\ell}^0$ is radial in the following sense: it is invariant under the natural action of the group $O(d) \times \{ 1 \}$. Therefore $\SSS Y_{2\ell}^0$, which is a function on the Grassmanian $\GG_{k+1,d+1}$, can be seen as a function on $\GG_{k+1,d+1}/O(d) = \Ss^1/\{\pm 1\}$. It implies that $\SSS$, while restricted to the space
\begin{equation*}
\bigoplus_{\ell=0}^\infty \spann (Y^0_{2\ell})
\end{equation*}
can be seen as an operator acting on even functions on $\Ss^1$. Moreover, using rotation invariance again, there exists $\mu_\ell$ such that with this identification,
\begin{equation*}
\SSS Y_{2\ell}^0 = \mu_\ell Y_{2\ell}^0.
\end{equation*}
The constant $\mu_\ell$ is thus given by
\begin{equation*}
\mu_\ell = \dfrac{(\SSS Y_{2\ell}^0)(e_{d+1})}{Y_{2\ell}(e_{d+1})}
\end{equation*}
Note that $(\SSS Y_{2\ell}^0)(e_{d+1})$ and $Y_{2\ell}^0(e_{d+1})$ are related to the Gegenbauer polynomial (see \cite{Samko}, Section $1.2$) $C_{2\ell}^{(d-1)/2}$ of degree $2\ell$ by
\begin{equation*}
(\SSS Y_{2\ell}^0)(e_{d+1}) = \left( \SSS \1 \right)\matrice{2\ell + d-2 \\ 2 \ell} C_{2\ell}^{(d-1)/2}(0), \ \ Y_{2\ell}^0(e_{d+1}) = \matrice{2\ell + d-2 \\ 2 \ell} C_{2\ell}^{(d-1)/2}(1).
\end{equation*}
Using $\SSS \1 = A_\SSS$ this gives
\begin{equation*}
\mu_\ell = A_\SSS \dfrac{C_{2\ell}^{(d-1)/2}(0)}{C_{2\ell}^{(d-1)/2}(1)}
\end{equation*}
The Gegenbauer polynomials are defined by their generating functional
\begin{equation*}
\left(1-2tx + x^2 \right)^{-\frac{d-1}{2}} = \sum_{m=0}^\infty C_m^{(d-1)/2} (t) x^m.
\end{equation*}
Evaluating this at $t=0,1$ gives:
\begin{equation*}
\left(1 + x^2 \right)^{-\frac{d-1}{2}} = \sum_{m=0}^\infty C_m^{(d-1)/2} (0) x^m, \ \ \left(1-x\right)^{-(d-1)} = \sum_{m=0}^\infty C_m^{(d-1)/2} (1) x^m.
\end{equation*}
On the other hand,
\begin{equation*}
\left(1 + x^2 \right)^{-\frac{d-1}{2}} = 1 - \dfrac{d-1}{2} x^2 + ... + \dfrac{(-1)^\ell}{2^\ell\ell!} \left( d-1 \right)\left( d+1 \right)....\left( d+ 2 \ell-3 \right) x^{2\ell} + o(x^{2\ell})
\end{equation*}
and 
\begin{equation*}
\left(1-x\right)^{-(d-1)} = 1 + (d-1) x + .... + \dfrac{1}{(2\ell)!} \left(d-1 \right)d\left( d +1 \right) ... \left(d + 2 \ell - 2 \right) x^{2\ell} + o(x^{2\ell}).
\end{equation*}
This implies:
\begin{equation*}
\mu_\ell = A_\SSS\dfrac{C_{2\ell}^{(d-1)/2}(0)}{C_{2\ell}^{(d-1)/2}(1)} = (-1)^\ell A_\SSS \dfrac{ (2\ell)!}{2^\ell \ell !} \dfrac{(d-1)(d+1)...(d+2\ell -3)}{(d-1)d(d+1)...(d+2 \ell -2)}.
\end{equation*}
We therefore get
\begin{equation*}
\left|\dfrac{\mu_{\ell+1}}{\mu_{\ell}}\right| = \dfrac{2\ell +1}{2\ell + d}.
\end{equation*}
This ratio is strictly less than $1$ since $d \geq 2$, implying that $\mu_\ell^2 = \lambda_\ell$ is decreasing.\\

$(iii)$ We have 
\begin{equation*}
\mu_0 = A_\SSS, \ \mu_1 = -\dfrac{A_\SSS}{d}, \ \mu_2 = \dfrac{3A_\SSS}{d(d+2)}.
\end{equation*}
Therefore
\begin{equation*}
\lambda_2 = \left(\dfrac{3A_\SSS}{d(d+2)}\right)^2
\end{equation*}
leading to 
\begin{equation*}
Q(Y_4^0) = \dfrac{q-1}{2A_\SSS^2} \lambda_2 - \dfrac{p-1}{2} = \dfrac{1}{2d} \left(\dfrac{3}{d+2}\right)^2 - \dfrac{d-k}{2(d+1)} \leq \dfrac{1}{2d} \left(\dfrac{3}{d+2}\right)^2 - \dfrac{1}{2(d+1)}.
\end{equation*}
This is negative for all $d \geq 2$.

Moreover, by the minmax principle, if $G \in T_\1 \FF$ then
\begin{equation*}
Q_\1(G) = \dfrac{q-1}{2A_\SSS^2} \lr{\LL G, G}_\HH - \dfrac{p-1}{2} |G|_\HH^2 \leq -\lambda_\SSS |G|_\HH^2.
\end{equation*}
This proves point $(iii)$.
\end{proof}

\section{About the second variation.}\label{sec:14}
As we said in the previous section our aim was to study the second variation associated to the function \eqref{eq:hZ}. Unfortunately this functional is not $C^2$ for $p \leq 2$, just because 
\begin{equation*}
t \mapsto |1+tG(x)|^p
\end{equation*}
is not $C^2$ when $G(x)$ is large. This is reflected in the fact that $L^p(\GG_{1,d+1})$ is not embedded in $\HH$ for $p < 2$. In order to go around this technical problem we use \cite{Christ5}, Proposition $15.1$, which we reformulate for our purpose here.

\begin{proposition}\label{prop:p2} Assume $p<2$. There exist some constant $c,C, t_0 >0$ and $\gamma>1$ with the following property. For $0 < t \leq t_0$ and for all $G \in L^p(\GG_{1,d+1})$ with 
\begin{equation*}
\int_{\GG_{1,d+1}} G = 0, \ \ \int_{\GG_{1,d+1}} G^p = 1,
\end{equation*}
split $G=G_\sharp+G_\flat$ with
\begin{equation*}
G_\sharp(x) = \system{ G(x)  & \ \text{if} \ |G(x)| \leq t^{-1+\gamma^{-1}}, \\  0 & \ \text{otherwise.} }
\end{equation*}
Then $G_\sharp \in \HH$ and
\begin{equation*}
\dfrac{|\SSS(\1+tG)|_q}{A_\SSS|\1+tG|_p} \leq 1 + t^2 Q_\1(G_\sharp) + Ct^{2+\gamma^{-1}} |G_\sharp|_p^2 - c t^{(2-p)\gamma^{-1}+p} |G_\flat|_p^p.
\end{equation*}
\end{proposition}

Now consider $G \in L^p(\GG_{1,d+1}) \cap T_\1\FF^\perp$, with norm $1$. Then
\begin{equation*}
Q_\1(G) \leq 1 + t^2 Q_\1(G_\sharp) + Ct^{2+\gamma^{-1}} |G_\sharp|_p^2 - c t^{(2-p)\gamma^{-1}+p} |G_\flat|_p^p.
\end{equation*}
Unfortunately, $G_\sharp$ does not need to belong to $T_\1\FF^\perp$, and so we cannot directly apply Lemma \ref{lem:boundQ}. Write instead
\begin{equation*}
G_\sharp  = G_\sharp^\perp + G_\sharp^\top,
\end{equation*}
with $G_\sharp^\perp \in T_\1 \FF^\perp$ and $G_\sharp^\top \in T_1\FF$. Then
\begin{equation*}
Q_\1(G_\sharp) = Q_\1(G_\sharp^\perp) + Q_\1(G_\sharp^\top).
\end{equation*}
We can apply Lemma \ref{lem:boundQ} to the first term. Let $V_i$, $1 \leq i \leq n$ be a basis of $T_\1\FF$. The space $T_\1 \FF$ is finite dimensional so there exists a constant $C'$ with
\begin{equation*}
Q_\1(G_\sharp^\top) \leq C'\sum_{i=1}^n \left|\int_{\GG_{1,d+1}} G_\sharp^\top V_i\right|^2.
\end{equation*}
This leads to
\begin{equation*}
Q_\1(G_\sharp) \leq Q_\1(G_\sharp^\perp) + C' \sum_{i=1}^n \left|\int_{\GG_{1,d+1}} G_\sharp^\top V_i\right|^2.
\end{equation*}
As $G_\sharp^\perp, G \in T_\1 \FF^\perp$ and $G_\sharp^\top = G - G_\sharp^\perp - G_\flat$, we get
\begin{align*}
Q_\1(G_\sharp) & \leq -\lambda_\SSS |G_\sharp|^2_\HH + C' \sum_{i=1}^n \left|\int_{\GG_{1,d+1}} G_\flat V_i\right|^2 \\
    & \leq -\lambda_\SSS |G_\sharp|^2_\HH + C' |G_\flat|^2_p
\end{align*}
In the last inequality we used H\"older's inequality and the fact that $T_\1\FF \subset L^\infty(\GG_{1,d+1})$. Group this inequality with the conclusion of Proposition \ref{prop:p2} to get that for $t \leq t_0$,
\begin{equation*}
\dfrac{|\SSS(\1+tG)|_q}{A_\SSS|\1+tG|_p} \leq 1 - \lambda_\SSS t^2 |G_\sharp|^2_\HH + Ct^{2+\gamma^{-1}} |G_\sharp|_p^2 - c t^{(2-p)\gamma^{-1}+p} |G_\flat|_p^p  + C' t^2 |G_\flat|^2_p.
\end{equation*}
We use this inequality in the next Session.

\section{Proof of Theorems \ref{thm:global} and \ref{thm:local}.}\label{sec:15}

As we are working in a non-Hilbertian case we define the linearized normal space $\tT_\1\FF^\perp \subset L^p(\GG_{1,d+1})$ as
\begin{equation*}
\tT_\1 \FF^\perp = \left\{ G \in L^p(\GG_{1,d+1}), \ \int_{\GG_{1,d+1}} GH^{p-1} = 0 \ \text{for all} \ H \in T_\1 \FF \right\}. 
\end{equation*}

\begin{lem}\label{lem:orth} There exists $\delta > 0$ such that if 
\begin{equation*}
d(F,\FF) := \inf_{H \in \FF} |H-F|_{L^p(\GG_{1,d+1})} \leq \delta |F|_p,
\end{equation*}
then we can decompose uniquely $F = H + G$, where $H \in \FF$ and $G \in \tT_H \FF^\perp$. 
\end{lem}

\begin{proof} We note that first it suffices to prove the Lemma when $|F|_p=1$, and second that we can assume that $\1$ minimizes the distance between $F$ and $\FF$ -- using that extremizers are unique modulo symmetries. Consider the map
\begin{equation*}
\begin{matrix}
 \Phi & : & \FF \times \tT_\1 \FF^\perp  &  \rightarrow &  L^p(\GG_{1,d+1}) \\
      &   & (H,G) & \mapsto & G+H.
\end{matrix}
\end{equation*}
$\Phi(\1,0) = \1$ and $d\Phi_{(\1,0)}= \Id$. Therefore every function $F$ with $|F-\1|_p$ small enough can be written as $H+G$ where $H \in \FF$ and $G \in \tT_\1 \EE^\perp$.
\end{proof}

It should be noted that $H$ does not minimize the distance between $F$ and extremizers. We will define for later use 
\begin{equation*}
d^\star(F,\FF) = |G|_p.
\end{equation*}
when $F$ satisfies the assumptions of Lemma \ref{lem:orth}. It is clear that $d^\star(F,\FF) \geq d(F,\FF)$. Moreover, because of the proof of Lemma \ref{lem:orth}, the reverse inequality is almost true: if $d(F,\FF)$ is small enough, there exists a constant $C$ such that
\begin{equation*}
C d(F,\FF) \geq d^\star(F,\FF).
\end{equation*}

Now we are ready to prove the Theorems.

\begin{proof} To prove Theorem \ref{thm:local} it suffices to prove Theorem \ref{thm:local2}. Consider $F$ with $|F|_p=1$ and $d(F,\FF) \leq \delta$. Write $F=H+tG$ with $H \in \FF$, $G \in \tT_{H}\FF^\perp$ with $|G|_p = 1$, and $t=d^\star(G,\FF)$.  Using the group of symmetries we can assume that $H = \1$ and $G \in \tT_\1\FF^\perp$

We start by the case $p \geq 2$. In this case Section \ref{sec:14} is not needed as the function \eqref{eq:hZ} is $C^2$ near $0$. Therefore for some $\epsi_0 >0$ depending only on $p,q$,
\begin{equation*}
\dfrac{|\SSS F|_q}{A_\SSS} = 1 + t^2 Q_\1(G) + o(t^{2+\epsi_0}) \leq 1 - \lambda_\SSS d^\star(F,\FF)^2+ o \left(d^\star(F,\FF)^{2+\epsi_0} \right).
\end{equation*}
This proves the local version.

Now let us treat the case $p < 2$. We have
\begin{align*}
\dfrac{|\SSS F|_q}{A_\SSS} & = \dfrac{|\SSS \phi \star(\1+t G)|_q }{A_\SSS|\phi \star(\1+t G)|_p} = \dfrac{|\SSS (\1+t G)|_q}{A_\SSS|\1+t  G|_p} \\
  & \leq 1 - \lambda_\SSS t^2 |G_\sharp|^2_\HH + Ct^{2+ \gamma^{-1}} |G_\sharp|_p^2 - c t^{(2-p)\gamma^{-1}} t^p |G_\flat|_p^p  + C' t^2 |G_\flat|^2_p.
\end{align*}
Note that as $p <2$, $\HH \hookrightarrow L^p(\GG_{1,d+1})$. Therefore, 
\begin{equation*}
|G_\sharp|_\HH \geq |G_\sharp|_p \geq |G|_p - |G_\flat|_p = 1 - |G_\flat|_p.
\end{equation*}
This implies 
\begin{equation*}
\dfrac{|\SSS F|_q}{A_\SSS} \leq 1 - \lambda_\SSS t^2 + Ct^{2+ \gamma^{-1}} |G|_p^2 - c t^{(2-p)\gamma^{-1}} t^p |G_\flat|_p^p  + (C'+\lambda_\SSS) t^2 |G_\flat|^2_p.
\end{equation*}
Note that as $1 < p < 2$ and $\gamma > 1$, $t^2 = O(t^{(2-p)\gamma^{-1}})$. As a consequence if $t$ is small enough the bad term $(C'+\lambda_\SSS) t^2 |G_\flat|^2_p$ can be entirely absorbed by the good term $- c t^{(2-p)\gamma^{-1}} t^p |G_\flat|_p^p$, leading to 
\begin{equation*}
\dfrac{|\SSS F|_q}{A_\SSS} \leq 1 - \lambda_\SSS t^2 + O(t^{2+\epsi}) = 1 - \lambda_\SSS d^\star(F,\FF)^2 + o\left( d^\star(F,\FF)^{2+\epsi_0} \right)
\end{equation*}
for some $\epsi_0>0$. This concludes the proof of Theorem \ref{thm:local2} and thus of Theorem \ref{thm:local}.

The proof of Theorem \ref{thm:global} follows from the standard arguments of \cite{BiaEgn}. We prove it for $k=d-1$. The same exact proof applies for $k \neq d-1$ with the radial extra-assumption. Then there exists a sequence $F_n$ such that
\begin{equation*}
|F_n|_p = 1, \ \ \dfrac{|\SSS F_n|_q}{A_\SSS} \geq 1 - \dfrac{1}{n} d(F_n,\FF)^2.
\end{equation*}
As $d(F_n,\FF) \leq |F_n|_p \leq 1$, we can assume that $d(F_n,\FF)$ converges to some value $\az \in [0,1]$. 

If $\az=0$, then using \cite{Christ4} there must exists a sequence of symmetries $\phi_n$ so that $\phi_n \star F_n$ converges to $\1$. If $n$ is large enough, then by Theorem \ref{thm:local2}
\begin{align*}
1 - \dfrac{1}{n} d(F_n,\FF)^2 & \leq \dfrac{|\SSS F_n|_q}{A_\SSS} \\  & = \dfrac{|\SSS \phi_n \star F_n|_q}{A_\SSS} \leq 1 - \lambda_\SSS d^\star(\phi_n \star F_n,\FF)^2 + o\left( d^\star(\phi_n \star F_n,\FF)^{2+\epsi_0} \right).
\end{align*}
Recall that as $\phi_n \star F_n$ goes to $\1$, $d^\star(\phi_n \star F_n, \FF)$ is comparable to $d(\phi_n \star F_n,\FF)$. Thus $0 < \lambda_\SSS \leq n^{-1}$, which is a contradiction.\\

If $\az > 0$, then
\begin{equation*}
\dfrac{|\SSS F_n|_q}{A_\SSS} \geq 1 - \dfrac{d(F_n,\FF)}{n} \rightarrow 1.
\end{equation*}
Therefore $F_n$ is an extremizing sequence. But then after suitable rescaling it converges to an extremizer, which is not possible as $d(F_n,\FF) \rightarrow \az$.
\end{proof}

\noindent \textbf{Remark:} \textit{(Sharpness of the constant $\lambda_\SSS$).} The constant $\lambda_\SSS$ in Theorem \ref{thm:local2} is optimal. Indeed, if $Y_4^0$ denotes the zonal fourth order spherical harmonic seen as a function on $\GG_{1,d+1}$ then $\1 + t Y_4^0 \in L^p(\GG_{1,d+1})$ and
\begin{equation*}
\dfrac{|\SSS (\1+t Y_4^0)|_q}{A_\SSS|\1+t Y_4^0|_p} = 1 - t^2 Q(Y_4^0) + o(t^2) = 1 - \lambda_\SSS t^2 + o(t^2).
\end{equation*}
Note that we do not know the optimal value of the constant $\lambda_\RRR$, even in the case where $d(f,\EE)$ is replaced by $d^\star(f,\EE)$. This comes from the fact that
\begin{equation*}
d^\star(f,\EE) \neq d^\star(\Ii f,\FF).
\end{equation*}
This is related to the fact that there is \textit{a priori} no isometric connection between $\SSS$ and $\RRR$ when $p,q$ are not given by \eqref{eq:pq}.

\vspace{0.5cm}
\begin{center}
\noindent
{\sc  Appendix 1: {Compactness of the restricted operator $\TT_R$}} 
\end{center}
\vspace{0.4cm}
\renewcommand{\theequation}{A.\arabic{equation}}
\refstepcounter{section}
\renewcommand{\thesection}{A}
\setcounter{equation}{0}

Here we want to prove lemma \ref{lem:comp}. The operator $\TT_R$ is formally defined as $\TT_R := \TT \1_{[0,R]}$ and maps $L^\infty([0,R])$ to itself. Moreover, since $L^\infty([0,R]) \hookrightarrow L^q([0,R])$, we just have to prove that $\TT_R : L^\infty([0,R]) \rightarrow L^\infty([0,R])$ is compact.\\

Compactness in $L^\infty([0,R])$ is expressed through the standard Arzel\`a-Ascoli theorem: 

\begin{theorem}
Let $\PP \subset L^\infty([0,R])$. Then $\PP$ is relatively compact if and only if it is bounded in $L^\infty([0,R])$ and equicontinuous.
\end{theorem}

Thus we want to show that $\PP := \TT (\{ f \in L^\infty([0,R], | f |_\infty \leq 1 \})$ is a equicontinuous family of functions in $L^\infty$. Let $f \in \PP$. Then for all $0\leq r \leq r+h \leq R$,
\begin{align*}
|\TT f(r+h) - \TT f(r)| & \leq \disp{\int_0^R |f(u)| |\1_{u \geq r+h}(u^2-(r+h)^2)^{k/2-1}- \1_{u \geq r}(u^2-r^2)^{k/2-1}|udu} \\
   & \leq  | f |_\infty \disp{\int_0^R |\1_{u \geq r+h}(u^2-(r+h)^2)^{k/2-1}- \1_{u \geq r}(u^2-r^2)^{k/2-1}|udu }\\
   & := | f |_\infty I(h,r).
\end{align*}
We just want to estimate $I(h,r)$.
\begin{align*}
I(h,r) & \leq \disp{ \int_{r+h}^R | (u^2-(r+h)^2)^{k/2-1} - (u^2-r^2)^{k/2-1} | udu }\\ 
& \hspace{5 mm} + \disp{\int_0^R |\1_{u \geq r+h} - \1_{u \geq r} | (u^2-r^2)^{k/2-1} udu  }.
\end{align*}
With the change of variable $u^2 = v$, this gives
\begin{align*}
I(h,r) & \lesssim \disp{ \int_{(r+h)^2}^{R^2} | (v-(r+h)^2)^{k/2-1} - (v-r^2)^{k/2-1} | dv }\\ 
& \hspace{5 mm} + \disp{\int_{0}^{R^2} |\1_{v \geq (r+h)^2} - \1_{v \geq r^2} | (v-r^2)^{k/2-1} dv  } \\
 & \lesssim \disp{ \int_{(r+h)^2}^{R^2} | (v-(r+h)^2)^{k/2-1} - (v-r^2)^{k/2-1} | dv }\\ 
& \hspace{5 mm} + \disp{\int_{r^2}^{(r+h)^2} (v-r^2)^{k/2-1} dv  }.
\end{align*}
By Holder's inequality,
\begin{equation*}
\disp{\int_{r^2}^{(r+h)^2} (v-r^2)^{k/2-1} dv  } \leq h^{1/3} \left(\int_{r^2}^{R^2} \left((v-r^2)^{k/2-1}\right)^{3/2} dv \right)^{2/3}
\end{equation*}
and this converges to $0$ as $h$ converges to $0$. An estimate for the other term is 
\begin{align*}
\disp{ \int_{(r+h)^2}^{R^2}  | (v-(r+h)^2)^{k/2-1} -} & (v-r^2)^{k/2-1} | dv  \leq \disp{ \int_{(r+h)^2}^{R^2} \int_{r^2}^{(r+h)^2} \lv \dfrac{k}{2} - 1  \rv s |v-s|^{k/2-2} ds dv } \\
& \lesssim \disp{  \int_{r^2}^{(r+h)^2} \int_{(r+h)^2}^{R^2} s |v-s|^{k/2-2} dv ds }  \\
& \lesssim \disp{ \int_{r^2}^{(r+h)^2} s ((R^2-s)^{k/2-1} - ((r+h)^2-s)^{k/2-1}) ds } .
\end{align*}
Applying H\"older's inequality $\PP$ is equicontinuous and $\TT_R$ is compact.


\begin{thebibliography}{2}
\bibitem[BaeLos]{BaeLos} Baernstein, Albert, II and Loss, Michael. Some conjectures about $L^p$ norms of $k$-plane transforms. Rend. Sem. Mat. Fis. Milano 67 (1997), 9–26 (2000). 

\bibitem[BiaEgn]{BiaEgn}  Bianchi, Gabriele and Egnell, Henrik. A note on the Sobolev inequality. J. Funct. Anal. 100 (1991), no. 1, 18-24.

\bibitem[CheFraWet]{CheFraWet}  Chen, Shibing, Frank, Rupert L. and Weth, Tobias. Remainder terms in the fractional Sobolev inequality. Indiana Univ. Math. J. 62 (2013), no. 4, 1381-1397.

\bibitem[CiaFusMagPra]{CiaFusMagPra}  Cianchi, A., Fusco, N., Maggi, F. and Pratelli, A. The sharp Sobolev inequality in quantitative form. J. Eur. Math. Soc. (JEMS) 11 (2009), no. 5, 1105-1139. 

\bibitem[Christ1]{Christ1}  Christ, Michael. Estimates for the $k$-plane transform. Indiana Univ. Math. J. 33 (1984), no. 6, 891-910.

\bibitem[Christ2]{Christ2} Christ, Michael. Quazi-extremizers for a Radon-like transform. Preprint, arXiv:1106.0722.

\bibitem[Christ3]{Christ3} Christ, Michael. On extremals for a Radon-like transform. Preprint, arXiv:1106.0728.

\bibitem[Christ4]{Christ4} Christ, Michael. Extremizers of a Radon transform inequality. Preprint, arXiv:1106.0719.

\bibitem[Christ5]{Christ5} Christ, Michael. A sharpened Hausdorff-Young inequality. Preprint, arXiv:1406.1210.

\bibitem[Drouot]{Drouot} Drouot, Alexis. Sharp constant for a $k$-plane transform inequality. Analysis and PDE (2014), vol. 7, no 6, 1237-1252. 

\bibitem[Drury1]{Drury}  Drury, S. W. Generalizations of Riesz potentials and $L^p$ estimates for certain $k$-plane transforms. Illinois J. Math. 28 (1984), no. 3, 495-512.

\bibitem[Drury2]{Drury1}  Drury, S. W. $L^p$ estimates for certain generalizations of $k$-plane transforms. Illinois J. Math. 33 (1989), no. 3, 367-374.


\bibitem[Flock]{Flock} Flock, Taryn. Uniqueness of extremizers for an endpoint inequality of the $k$-plane transform. Preprint, arXiv:1307.6551.

\bibitem[Grafakos]{Grafakos} Grafakos, Loukas. Classical and Modern Fourier Analysis, Prentice Hall, NJ 2003.

\bibitem[Lieb]{Lieb} Lieb, Elliot. Sharp constants in the Hardy-Littlewood-Sobolev and related inequalities, Ann. Math. 118
(1983), 349-374.

\bibitem[Lions]{lions} Lions, Pierre-Louis. The concentration-compactness principle in the Calculus of Variations. The locally compact case, Part 1. Ann. I. H., Anal. Nonlin., 1 (1984), 109-145.

\bibitem[Mattila]{Mattila} Mattila, Pertti. Geometry of sets and measures in Euclidean spaces. Fractals and rectifiability. Cambridge Studies in Advanced Mathematics, 44. Cambridge University Press, Cambridge, 1995.


\bibitem[ObeSte]{ObeSte} Oberlin, D. M. and Stein, E. M. Mapping properties of the Radon transform. Indiana Univ. Math. J. 31 (1982), no. 5, 641-650. 

\bibitem[Samko]{Samko}  Samko, Stefan G. Hypersingular integrals and their applications. Analytical Methods and Special Functions, 5. Taylor \& Francis, Ltd., London, 2002.


\end{thebibliography}
\end{document}